\documentclass[12pt]{article}   
   
\usepackage{latexsym,amsfonts,amsmath,epsfig,tabularx,amsthm,amssymb,enumitem,bm}
\usepackage{float}

\usepackage{microtype}

\usepackage[draft=false,colorlinks,bookmarksnumbered,linkcolor=black, citecolor=black]{hyperref}

\usepackage{aliascnt}

\newtheorem{theorem}{Theorem}[section]

\newaliascnt{lem}{theorem}
\newtheorem{lemma}[lem]{Lemma}

\aliascntresetthe{lem}

\newaliascnt{ass}{theorem}

\aliascntresetthe{ass}

\newaliascnt{prop}{theorem}
\newtheorem{prop}[prop]{Proposition}

\aliascntresetthe{prop}

\newaliascnt{cor}{theorem}
\newtheorem{corollary}[cor]{Corollary}

\aliascntresetthe{cor}

\newaliascnt{defi}{theorem}
\newtheorem{defi}[defi]{Definition}

\aliascntresetthe{defi}

\theoremstyle{definition}
\newaliascnt{ex}{theorem}
\newtheorem{example}[ex]{Example}

\aliascntresetthe{ex}

\newaliascnt{rem}{theorem}
\newtheorem{remark}[rem]{Remark}

\aliascntresetthe{rem}

\renewcommand{\d}{\,\mathrm{d}}											
\renewcommand*{\epsilon}{\varepsilon}                                   
\renewcommand*{\rho}{\varrho}                                   		
\newcommand*{\nach}{\rightarrow}                                        
\newcommand*{\sep}{\; \vrule \;}                                        
\newcommand*{\N}{\mathbb{N}}                                            
\newcommand*{\R}{\mathbb{R}}                                            
\newcommand*{\C}{\mathbb{C}}                                            
\newcommand*{\Z}{\mathbb{Z}}                                            
\newcommand*{\B}{\mathcal{B}}                                           
\newcommand*{\M}{\mathcal{M}}                                         	
\newcommand*{\V}{\mathcal{V}}                                         	
\newcommand*{\T}{\mathcal{T}}                                         	
\newcommand*{\J}{\mathcal{J}}                                         	
\newcommand*{\leer}{\emptyset}                                          

\newcommand*{\abs}[1]{\left| #1 \right|}                                
\newcommand*{\norm}[1]{\left\| #1 \right\|}                             
\newcommand*{\ceil}[1]{\left\lceil #1 \right\rceil}                     
\newcommand*{\link}[1]{(\ref{#1})}                                      
\newcommand*{\distrmod}[2]{\left\langle\!\left\langle #1, #2 \right\rangle\!\right\rangle}
\newcommand*{\distr}[2]{\left\langle #1, #2 \right\rangle}              
\newcommand*{\dist}[2]{\mathrm{dist}\!\left( #1, #2 \right)}            

\renewcommand{\tilde}[1]{ \widetilde{#1} }        						
\DeclareMathOperator{\supp}{supp}										
\DeclareMathOperator{\spann}{span}										
\newcommand{\esssup}{ \mathop{\mathrm{esssup}}\limits }
\DeclareMathOperator{\diam}{diam}										
\DeclareMathOperator{\Id}{Id}											
\DeclareMathOperator{\ad}{ad}											

\usepackage{footmisc}

\usepackage{tikz}
\usetikzlibrary{arrows,decorations.pathmorphing,backgrounds,positioning,fit,petri}

\usepackage[margin=1.5cm]{caption}
\setlist{itemsep=2pt, topsep=2pt}

\title{Almost diagonal matrices and\\Besov-type spaces based on wavelet expansions \footnote{This work has been supported by Deutsche Forschungsgemeinschaft DFG (DA 360/19-1).}}     
 
\author{Markus Weimar\thanks{Philipps-University Marburg, Faculty of Mathematics and Computer Science, Hans-Meerwein-Stra{\ss}e, Lahnberge, 35032 Marburg, Germany. Email: weimar@mathematik.uni-marburg.de}} 

\begin{document}   
\maketitle

\begin{abstract}
\noindent
This paper is concerned with problems in the context of the theoretical foundation of adaptive (wavelet) algorithms for the numerical treatment of operator equations. 
It is well-known that the analysis of such schemes naturally leads to function spaces of Besov type. But, especially when dealing with equations on non-smooth manifolds, the definition of these spaces is not straightforward. 
Nevertheless, motivated by applications, recently Besov-type spaces $B^\alpha_{\Psi,q}(L_p(\Gamma))$ on certain two-dimensional, patchwise smooth surfaces were defined and employed successfully.
In the present paper, we extend this definition (based on wavelet expansions) to a quite general class of $d$-dimensional manifolds and investigate some analytical properties (such as, e.g., embeddings and best $n$-term approximation rates) of the resulting quasi-Banach spaces.
In particular, we prove that different prominent constructions of biorthogonal wavelet systems $\Psi$ on domains or manifolds $\Gamma$ which admit a decomposition into smooth patches actually generate the same Besov-type function spaces $B^\alpha_{\Psi,q}(L_p(\Gamma))$, provided that their univariate ingredients possess a sufficiently large order of cancellation and regularity (compared to the smoothness parameter $\alpha$ of the space).
For this purpose, a theory of almost diagonal matrices on related sequence spaces $b^\alpha_{p,q}(\nabla)$ of Besov type is developed. 


\end{abstract}

\section{Motivation and background}
During the past years, wavelets have become a powerful tool in pure and applied mathematics.
Especially for the numerical treatment of (elliptic) operator equations (which arise in models for various problems of modern sciences, e.g., in the context of image analysis, or signal processing) wavelet-based algorithms turned out be very efficient. In these schemes wavelets are the method of choice to discretize the partial differential or integral equation under consideration such that finally an approximate solution is obtained by solving a series of finite-dimensional linear systems involving only sparsly populated matrices.
Roughly speaking, this sparsity is accomplished by compression strategies which heavily exploit the multiscale structure of the wavelets on the one hand, as well as their attractive analytical properties (concerning support, cancellation, and smoothness) on the other hand.
In practice, often an additional gain of performance is observed when using \emph{adaptive} refinement and coarsening strategies that rely on (local) estimates of the residuum. 
Meanwhile, for many problems, this higher rate of convergence of adaptive methods (compared to non-adaptive schemes based on uniform refinement) can be justified also from the theoretical point of view.
However, large problem classes remain for which a solid mathematical foundation for the need of adaptive algorithms still has to be developed.

Clearly, the best what can be expected from an adaptive wavelet solver is that it realizes (at least asymptotically) the rate of the best $n$-term (wavelet) approximation to the unknown solution, since this natural benchmark describes the smallest error \emph{any} (non-linear) method can achieve using at most $n$ degrees of freedom.
When measuring the approximation error in the norm of $L_2$, or of Sobolev Hilbert spaces $H^s$, the correct smoothness spaces for calculating best $n$-term rates in the context of adaptive algorithms are given by (shifted versions of) the so-called \emph{adaptivity scale of Besov spaces} $B^\alpha_\tau(L_\tau)$, where $\tau^{-1}=\alpha/d+1/2$. 
Surprisingly enough, for a large class of problems including, e.g., elliptic PDEs and boundary integral equations, there exist adaptive wavelet schemes which obtain these orders of convergence, while the number of arithmetic operations stays proportional to the number of unknowns; see, e.g., \cite{CDD01, CDD02, DahRaaWer+2007, DahHarSch2006, DahHarSch2007}.
In contrast, the performance of non-adaptive algorithms is governed by the maximal smoothness of the unknown solution in the Sobolev scale $H^s$; cf.~\cite{DahDahDeV1997, DeV1998}. 
For many practical problems this regularity is limited due to singularities caused by the shape of the underlying domain. 
On the other hand, particularly for elliptic PDEs on Euclidean domains, it is known that the influence of these singularities on the maximal Besov regularity is significantly smaller; see, e.g., \cite{DahDeV1997, DahDieHar+2014, Han2012}. 
Therefore we can state that, at least for such problems, adaptivity pays off and their theoretical analysis naturally leads to function spaces of Besov type.

In the realm of operator equations defined on manifolds (especially for problems formulated in terms of integral equations at the boundary of some non-smooth domain) we are faced with additional, quite serious problems: the construction of suitable wavelet systems, on the one hand, and the investigation of the relevant function spaces, on the other hand.
Meanwhile, for geometries that admit a decomposition into smooth patches (e.g., stemming from models in Computer Aided Geometric Design), a couple of wavelet bases are known which perform quite well \cite{CTU00, CohMas2000, DS99, DahSch1999, HarSch2004, HS06}. Hence, the first issue has been solved satisfactory, at least for practically important cases. The aim of this paper is to shed some light on the second problem, because the picture here is not as complete.

In classical function space theory Besov spaces are defined on the whole of $\R^d$, e.g., by Fourier techniques. Then functions (or distributions) from these spaces can be simply restricted to $d$-dimensional domains $\Omega$ of interest. In all practically relevant cases this definition then coincides with intrinsic descriptions given, e.g., in terms of moduli of smoothness; cf.\ \cite{T06,T08,T83}.
Accepting the fact that wavelet characterizations (restricted from $\R^d$ to~$\Omega$) might involve a few wavelets whose support exceed the boundary of the domain, this method provides a handy approach towards regularity studies in Besov spaces for the case of such sets.
The definition and analysis of corresponding function spaces on general manifolds is more critical: 
There usually trace operators or sufficiently smooth pullbacks of local (overlapping) charts are used. Unfortunately, no intrinsic characterizations for trace spaces on complicated geometries are available and it is unclear how a wavelet characterization of these spaces should look like. 
Moreover, following the second approach, the maximal regularity of the resulting spaces would be naturally restricted by the \emph{global} smoothness of the manifold under consideration.
Therefore, in \cite{DahWei2013} we proposed and successfully exploited a completely different method to define higher-order Besov-type spaces $B^{\alpha}_{\Psi,q}(L_p(\Gamma))$ on (specific, two-dimensional, closed) manifolds $\Gamma$ which are only \emph{patchwise} smooth.
The idea is rather simple, but quite effective: Since we like to employ wavelets for our approximation schemes, only the decay of the wavelet coefficients of the object to be approximated is important. 
Consequently, in our spaces we collected all those functions whose coefficients w.r.t.\ \emph{one fixed} (biorthogonal) wavelet basis $\Psi$ exhibit a certain rate of decay which would be expected from a classical wavelet characterization. 
Although, from the application point of view, a definition like this is completely justified, it has a theoretical drawback: The spaces constructed this way
formally depend on the chosen wavelet system. 
In \cite[Rem.~4.2(ii)]{DahWei2013} it was stated that there are good reasons to assume that spaces built up on wavelets $\Psi$ with ``similiar'' properties actually coincide, at least in the sense of equivalent (quasi-)norms.
The main purpose of the current article is the verification of this conjecture for a large range of parameters and three classes of wavelet bases ($\Psi_{\mathrm{DS}}$, $\Psi_{\mathrm{HS}}$, and $\Psi_{\mathrm{CTU}}$) which are widely used in practice; see \autoref{thm:equivalence}.

Our material is organized as follows: \autoref{sect:Seq} exclusively deals with (operators on) sequence spaces $b^\alpha_{p,q}(\nabla)$ which later on will be crucial for the  definition of our function spaces $B^{\alpha}_{\Psi,q}(L_p(\Gamma))$ of Besov-type.
They are indexed by what we call \emph{multiscale grids}~$\nabla$, i.e., by sets which are tailored for the use in the context of multiresolution expansions on manifolds.
Furthermore, guided by the pioneer work of Frazier and Jawerth \cite{FJ90}, here we introduce classes $\ad\!\left(b^{\alpha_0}_{p,q}(\nabla^0),b^{\alpha_1}_{p,q}(\nabla^1) \right)$ of \emph{almost diagonal} matrices whose entries decay sufficiently fast apart from the diagonal. 
The main result of this first part (which is of interest on its own, but also will be essential in what follows) then is given by \autoref{thm:ad}. 
It states that every such matrix induces a bounded linear operator between the Besov-type sequence spaces $b^{\alpha_i}_{p,q}(\nabla^i)$, $i=0,1$, under consideration.
The remaining part of the paper is concerned with function spaces. 
In \autoref{sect:Fun} we recall what is meant by patchwise smooth geometries $\Gamma$ and we review basic concepts from multiscale analysis. 
Additionally, here we describe the three biorthogonal wavelet constructions $\Psi=(\Psi^{\Gamma}, \tilde{\Psi}^\Gamma)$ we are mainly interested in:
\begin{itemize}
	\item[1.)] Composite wavelet bases $\Psi=\Psi_{\mathrm{DS}}$ introduced by Dahmen and Schneider \cite{DS99} (for general operator equations),
	\item[2.)] Modified composite wavelets $\Psi=\Psi_{\mathrm{HS}}$ due to Harbrecht and Stevenson \cite{HS06} (which are the first choice in the so-called boundary element method for integral equations), and
	\item[3.)] Bases $\Psi=\Psi_{\mathrm{CTU}}$ (primarily used in the wavelet element method) established by Canuto, Tabacco, Urban \cite{CTU99, CTU00}.
\end{itemize} 
Afterwards, in \autoref{sect:Besov}, we extend the definition of Besov-type function spaces $B^{\alpha}_{\Psi,q}(L_p(\Gamma))$ given in \cite{DahWei2013} to a fairly general setting which particularly covers spaces on decomposable manifolds $\Gamma$ that are needed for practical applications.
Moreover, here we investigate some of the theoretical properties of these scales such as embeddings, interpolation results, and best $n$-term approximation rates. 
In \autoref{thm:localized_composite}, \autoref{sect:change}, we employ the theory of almost diagonal matrices developed in \autoref{sect:Seq}, to derive sufficient conditions for continuous one-sided \emph{change of basis embeddings} 
\begin{equation*}
	B^{\alpha}_{\Psi,q}(L_p(\Gamma))\hookrightarrow B^{\alpha}_{\Phi,q}(L_p(\Gamma)),
	\qquad 0\leq \alpha < \alpha^*.
\end{equation*}
Finally, these embeddings then pave the way to state and prove \autoref{thm:equivalence} which constitutes the main result of this paper: the equivalence $B^{\alpha}_{\Psi,q}(L_p(\Gamma))=B^{\alpha}_{\Phi,q}(L_p(\Gamma))$ for $\Psi,\Phi\in\{\Psi_{\mathrm{DS}}, \Psi_{\mathrm{HS}}, \Psi_{\mathrm{CTU}}\}$.
The article is concluded with an appendix (\autoref{sect:appendix}) which contains auxiliary assertions, as well as some quite technical proofs.

\textbf{Notation:} For families $\{a_{\J}\}_{\J}$ and $\{b_{\J}\}_{\J}$ of non-negative real numbers over a common index set we write $a_{\J} \lesssim b_{\J}$ if there exists a constant $c>0$ (independent of the context-dependent parameters $\J$) such that
\begin{equation*}
	a_{\J} \leq c\cdot b_{\J}
\end{equation*}
holds uniformly in $\J$.
Consequently, $a_{\J} \sim b_{\J}$ means $a_{\J} \lesssim b_{\J}$ and $b_{\J} \lesssim a_{\J}$.
In addition, if not further specified, throughout the whole paper we will assume that $\Gamma$ denotes an arbitrary set endowed with some metric $\rho_\Gamma$. In view of the applications we have in mind, later on we will impose additional conditions on $\Gamma$.

\section{Sequence spaces and almost diagonal matrices}
\label{sect:Seq}
For every $\T\neq \leer$ let us define a \emph{pseudometric} on $\Gamma \times \T$ by setting
\begin{equation}\label{eq:pseudometric}
	\dist{(y,t)}{(y',t')} 
	:= \rho_\Gamma(y,y')
	\quad \text{for} \quad y,y'\in\Gamma \quad \text{and} \quad t,t'\in\T.
\end{equation}

\begin{defi}\label{ass:struct}
	Let $d\in\N$. 
	We say the sequence $\nabla:=(\nabla_j)_{j\in\N_0}$ is a \emph{multiscale grid of dimension $d$ (for $\Gamma$)} if there are absolute constants $c_1,c_2,c_3>0$ such that the following three assumptions are satisfied:
	\begin{enumerate}[label=(A\arabic{*}), ref=A\arabic{*}]
	\item \label{A1} For some finite index set $\T\neq \leer$ and all $j\in\N_0$ the set $\nabla_j$ forms a $(c_1\, 2^{-j})$-net for $\Gamma\times\T$ w.r.t.\ \link{eq:pseudometric}.
	\item \label{A2} $\nabla$ is uniformly well-separated, meaning that uniformly in $j\in\N_0$ it holds 
					\begin{equation*}
						\sup_{\xi\in\nabla_j}\#\{\xi'\in\nabla_j \;|\; \dist{\xi}{\xi'} \leq c_2\, 2^{-j} \} \lesssim 1.
					\end{equation*}
	\item \label{A3} $\nabla$ is uniformly $d$-dimensional, in the sense that uniformly in $j\in\N_0$ we have
					\begin{equation*}
						\sup_{\xi\in\nabla_j} \#\{\xi'\in\nabla_j \;|\; \dist{\xi}{\xi'} \leq c_3 \} \sim 2^{dj}.
					\end{equation*}
\end{enumerate}
\end{defi}

\begin{remark}
We note in passing that these assumptions clearly force $\Gamma$ to be $d$-dimensional (in a certain sense), as well.
Moreover, if the set $\Gamma$ is bounded, meaning that its diameter $\diam(\Gamma):=\sup_{y,y'\in\Gamma}\rho_\Gamma(y,y')$ is finite, then \autoref{ass:struct} implies that
\begin{enumerate}[label=\textit{(A4a)}, ref=A4a, leftmargin=*]
	\item \label{A4a} $\#\nabla_j \sim 2^{dj}$.
\end{enumerate}
Otherwise (if $\Gamma$ is unbounded), we necessarily have 
\begin{enumerate}[label=\textit{(A4b)}, ref=A4b, leftmargin=*]
	\item \label{A4b} $\#\nabla_j=\infty$ \quad \textit{for all} \quad $j\in\N_0$.\hfill$\square$
\end{enumerate}
\end{remark}

Typical examples of multiscale grids cover index sets related to expansions (w.r.t.\ certain building blocks such as wavelets, atoms, molecules, \ldots) in function spaces on $\Gamma=\R^d$, $d\in\N$. Indeed, when dealing with wavelet expansions, $\tilde{\nabla}_j \subseteq \Z^d \times \{1,\ldots,2^d-1\}$ usually is interpreted as index set encoding the location and type of all wavelets at level $j\in\N_0$.
The same reasoning also applies for (bounded) domains $\Gamma \subsetneq\R^d$. Obviously, every such sequence $\tilde{\nabla}=(\tilde{\nabla}_j)_{j\in\N_0}$ can be identified with some multiscale grid $\nabla$ in the above sense.
However, note that the index sets in \autoref{ass:struct} are designed in a way such that all indices are directly associated with some point \emph{in the domain~$\Gamma$} which allows to handle more complex situations, as well. 
If $d=2$, say, then we may also think of $\Gamma=\partial\Omega$ being the surface of some bounded polyhedral domain $\Omega\subset\R^3$, or an even more abstract (two-dimensional) manifold with or without boundary. 
Then for all $j\in\N_0$ the sets 
\begin{equation*}
	\{y \in\Gamma \sep (y,t)=\xi \in \nabla_j \text{ for some } t \in \T\}
\end{equation*}
yield discretizations of $\Gamma$. 
Let us mention that for infinitely smooth manifolds $\Gamma$ a similar approach has been proposed already in \cite{T08}.

Now we are well-prepared to introduce \emph{sequence spaces $b^\alpha_{p,q}(\nabla)$ of Besov type} on multiscale grids $\nabla$.

\subsection{Sequence spaces of Besov type}
The following \autoref{def:b} is inspired by sequence spaces which naturally arise in the context of classical (wavelet)  characterizations of function spaces; see, e.g., \cite[Def.~3]{DNS06}.
Here and in what follows we slightly abuse the notation and write $(j,\xi)\in\nabla=(\nabla_j)_{j\in\N_0}$ if $j\in\N_0$ and $\xi\in\nabla_j$.

\begin{defi}\label{def:b}
Let $0<p<\infty$, $0<q\leq \infty$, as well as $\alpha\in\R$, and let $\nabla=(\nabla_j)_{j\in\N_0}$ denote some multiscale grid of dimension $d\in\N$ in the sense of \autoref{ass:struct}. Then
\begin{itemize}
	\item[(i)] we define the sequence space $b^\alpha_{p,q}(\nabla) := \left\{ \bm{a}=(a_{(j,\xi)})_{(j,\xi)\in\nabla} \subset \C \sep \norm{\bm{a} \sep b^\alpha_{p,q}(\nabla)} < \infty \right\}$ endowed with the (quasi-)norm
\begin{equation}\label{def:seq_norm}
	\norm{\bm{a} \sep b^\alpha_{p,q}(\nabla)} 
	:= \begin{cases}
		\left(  \displaystyle \sum_{j=0}^\infty\nolimits 2^{ j \left( \alpha + d \big[1/2-1/p \big]\!\right)q } \left[ \sum_{\xi\in\nabla_j} \nolimits \abs{a_{(j,\xi)}}^p \right]^{q/p} \right)^{1/q}, & \text{ if } q<\infty,\\
		\displaystyle \sup_{j\in\N_0} 2^{ j \left( \alpha + d \big[1/2-1/p \big]\!\right) } \left[ \sum_{\xi\in\nabla_j}\nolimits \abs{a_{(j,\xi)}}^p \right]^{1/p}, & \text{ if } q=\infty.
	\end{cases}
		\end{equation}
	\item[(ii)] we set $\sigma_p := \sigma_p(d) := d \cdot \max\!\left\{\frac{1}{p}-1,0\right\}$.
\end{itemize}
\end{defi}

\begin{remark}\label{rem:props_b}
Some comments are in order:
\begin{itemize}
	\item[(i)] Using standard arguments it can be checked that $b_{p,q}^\alpha(\nabla)$ are always quasi-Banach spaces. They are Banach spaces if and only if $\min\{p,q\} \geq 1$, and Hilbert spaces if and only if $p=q=2$.
	\item[(ii)] For special choices of the parameters $p$, $q$, and $\alpha$, the (quasi-)norms defined in \link{def:seq_norm} simplify significantly. 
	In particular,
\begin{equation*}
	\norm{\bm{a} \sep b^0_{2,2}(\nabla)} 
	= \left( \sum_{(j,\xi)\in\nabla} \abs{a_{(j,\xi)}}^2 \right)^{1/2},
\end{equation*}
such that $b^0_{2,2}(\nabla)=\ell_2(\nabla)$ with equal norms. More general we have $b^{\alpha_\tau}_{\tau,\tau}(\nabla)=\ell_\tau(\nabla)$ with equal norms, where
\begin{equation}\label{eq:adapt}
	\tau = \left(\frac{\alpha_\tau}{d} + \frac{1}{2}\right)^{-1} 
	\quad \text{with} \quad
	\alpha_\tau \geq 0
\end{equation}
	defines the so-called \emph{adaptivity scale} w.r.t.\ $\ell_2(\nabla)$.
\hfill$\square$
\end{itemize} 
\end{remark}

Before we turn to operators acting on the sequence spaces just introduced, let us add the following embedding result which will be useful later on. Its proof is postponed to the appendix; see \autoref{sect:proofs}.

\begin{prop}[Standard embeddings]
\label{prop:seq-embedding}
Let $0<p_0,p_1<\infty$, as well as $0<q_0,q_1\leq\infty$, and $\alpha,\gamma\in\R$. 
Moreover, let $\nabla$ denote some multiscale grid of dimension $d\in\N$. 
If, in addition, condition \link{A4a} holds for $\nabla$ then the embedding
	\begin{equation*}
		b^{\alpha+\gamma}_{p_0,q_0}(\nabla) 
		\hookrightarrow b^{\alpha}_{p_1,q_1}(\nabla) 
	\end{equation*}		
	exists (as a set theoretic inclusion) 
	if and only if it is continuous if and only if one of the subsequent conditions applies:
	\begin{itemize}
		\item[$\bullet$)] $\gamma > d \cdot \max\!\left\{0, \frac{1}{p_0} - \frac{1}{p_1} \right\}$,
		\item[$\bullet$)] $\gamma = d \cdot \max\!\left\{0, \frac{1}{p_0} - \frac{1}{p_1} \right\}$ \quad and \quad $q_0 \leq q_1$.
	\end{itemize}
Furthermore, if $\nabla$ satisfies \link{A4b} instead of \link{A4a} then a corresponding characterization holds with the additional condition $p_0\leq p_1$.
\end{prop}

\begin{remark}
Based on the reduction arguments we used to prove \autoref{prop:seq-embedding} it would be possible to derive a lot of further results related to the spaces $b^{\alpha}_{p,q}(\nabla)$ such as, e.g., interpolation assertions or estimates for entropy numbers. 
As such properties are beyond the scope of the present paper, we will not follow this line of research here.
\hfill$\square$
\end{remark}

\subsection{Almost diagonal matrices}
Clearly, every linear mapping $M$ defined between sequence spaces (indexed by multiscale grids $\nabla^0$ and $\nabla^1$ on $\Gamma$, respectively) can be represented as the formal product with some double-infinite complex matrix 
\begin{equation*}
	\M := \{ m_{(j,\xi),(k,\eta) }\}_{(j,\xi)\in\nabla^1, (k,\eta)\in\nabla^0},
\end{equation*}
i.e., $M\colon \bm{a}\mapsto M\bm{a} := \left( (\M\bm{a})_{(j,\xi)} \right)_{(j,\xi)\in\nabla^1}$ with
\begin{equation*}
	(\M\bm{a})_{(j,\xi)} := \sum_{(k,\eta)\in\nabla^0} m_{(j,\xi),(k,\eta)}\, a_{(k,\eta)}, \qquad (j,\xi)\in\nabla^1.
\end{equation*}

We shall follow the ideas given in \cite[Sect.~3]{FJ90} and define classes of \emph{almost diagonal} matrices for the sequence spaces of Besov type established in \autoref{def:b}.
\begin{defi}\label{def:ad}
	Let $0<p<\infty$ and $0<q\leq \infty$. 
	Moreover, let $\nabla^0$ and $\nabla^1$ denote two multiscale grids of dimension $d\in\N$ for some set $\Gamma$.
	\begin{itemize}
		\item[(i)] For $\alpha_0,\alpha_1\in\R$ a matrix $\M=\{m_{(j,\xi),(k,\eta)}\}_{(j,\xi)\in\nabla^1,(k,\eta)\in\nabla^0}$ is called \emph{almost diagonal between $b^{\alpha_0}_{p,q}(\nabla^0)$ and $b^{\alpha_1}_{p,q}(\nabla^1)$} if there exists $\epsilon > 0$ such that
			\begin{equation}\label{cond_ad}
				\sup_{(j,\xi)\in\nabla^1,(k,\eta)\in\nabla^0} \frac{\abs{m_{(j,\xi),(k,\eta)}}}{\omega_{(j,\xi),(k,\eta)}(\epsilon)} < \infty,
			\end{equation}
			where
			\begin{equation*}
				\omega_{(j,\xi),(k,\eta)}(\epsilon) 
				:= 2^{k\alpha_0-j\alpha_1} \cdot \frac{\min\!\left\{ 2^{-(j-k)(d/2+\epsilon)}, 2^{(j-k)(d/2+\epsilon+ \sigma_p)} \right\}}{\left[ 1+ \min\!\left\{ 2^k, 2^j \right\} \dist{\xi}{\eta} \right]^{d+\epsilon+\sigma_p}}.
			\end{equation*}
			In this case we write $\M \in \ad\!\left(b^{\alpha_0}_{p,q}(\nabla^0),b^{\alpha_1}_{p,q}(\nabla^1) \right)$.
		\item[(ii)] When there is no danger of confusion, we shall write $\ad_{p}^{\alpha_0,\alpha_1}$ as a shortcut for the class $\ad\!\left(b^{\alpha_0}_{p,q}(\nabla^0),b^{\alpha_1}_{p,q}(\nabla^1) \right)$. Furthermore, if $\alpha_0=\alpha_1=\alpha \in \R$ then we set $\ad_{p}^{\alpha}:=\ad_{p}^{\alpha_0,\alpha_1}$.
	\end{itemize}
\end{defi}

Roughly speaking, a matrix $\M$ belongs to the class $\ad_{p}^{\alpha_0,\alpha_1}$ if its entries decay fast enough apart from the diagonal ($m_{(j,\xi),(j,\xi)}$). 
If the sets $\nabla^i_j$, $i\in\{0,1\}$, are interpreted as index sets for the location (and type) of all wavelets at level $j$ on $\Gamma$ then the matrix entries $m_{(j,\xi),(k,\eta)}$ need to be small for all wavelets (indexed by $(j,\xi)\in\nabla^1$ and $(k,\eta)\in\nabla^0$, respectively) which
\begin{itemize}
	\item[$\bullet$)] are supported far away from each other (then $\dist{\xi}{\eta}\gg 0$), or
	\item[$\bullet$)] correspond to very different levels (then $\abs{j-k}\gg 0$).
\end{itemize}
Note that quite similar matrix classes naturally appear in the context of compression and preconditioning strategies used by elaborated (adaptive) wavelet algorithms for operator equations (e.g., for Schwartz kernel problems). 
Without going into details, we like to mention the so-called \emph{Lemari\'{e} algebra} and refer to \cite{Coh2003, CDD01, CDD02, Dah1997,Lem1989} for details.
		
\begin{remark}\label{rem:ad}
Some further comments are in order:
\begin{itemize}
	\item[(i)] Observe that \link{cond_ad} is independent of the index $q$ which justifies to suppress this fine-tuning parameter in the abbreviations in \autoref{def:ad}(ii).
	\item[(ii)] Using the monotonicity of $\sigma_p$ (cf.~\autoref{def:b}(ii)), it is easily seen that
	\begin{equation*}
		\ad_{\widehat{p}}^{\alpha_0,\alpha_1} 
		\subseteq \ad_{p}^{\alpha_0,\alpha_1} 
		\subseteq \ad_{1}^{\alpha_0,\alpha_1}
		= \ad_{\tilde{p}}^{\alpha_0,\alpha_1}
		\qquad \text{for all} \qquad
		0 < \widehat{p} \leq p \leq 1 \leq \tilde{p} < \infty
	\end{equation*}
	and all $\alpha_0,\alpha_1\in\R$. That is, condition \link{cond_ad} is getting stronger when $1/p$ increases and it is independent of $p$ when $1/p\leq 1$.
	\hfill$\square$
\end{itemize}
\end{remark}

We are ready to state and prove the main result of this \autoref{sect:Seq}. 
It is inspired by \cite[Thm.~3.3]{FJ90} and shows that every almost diagonal matrix can be interpreted as a continuous linear operator on the sequence spaces introduced in \autoref{def:b}.

\begin{theorem}\label{thm:ad}
	Let $0<p<\infty$ and $0<q\leq \infty$, as well as $\alpha_0,\alpha_1\in\R$. 
	Moreover, let $\nabla^0$ and $\nabla^1$ denote two multiscale grids of dimension $d\in\N$ for some set $\Gamma$.	
	Then every matrix $\M\in \ad\!\left(b^{\alpha_0}_{p,q}(\nabla^0),b^{\alpha_1}_{p,q}(\nabla^1) \right)$ induces a bounded linear operator $M \colon b^{\alpha_0}_{p,q}(\nabla^0) \nach b^{\alpha_1}_{p,q}(\nabla^1)$.
\end{theorem}

Before we conclude this section by presenting a detailed derivation of this assertion, let us remark that \autoref{thm:ad} (as well as its proof) differs from \cite[Thm.~3.3]{FJ90} to some extend.
Indeed, in the current paper, we deal with sequence spaces which correspond to function spaces of Besov type, in contrast to Triebel-Lizorkin spaces discussed in \cite{FJ90}. 
Consequently, as we shall see, we can avoid the use of maximal inequalities due to Fefferman-Stein.
Another difference is that our notion of almost diagonal matrices (and thus also \autoref{thm:ad}) depends on two smoothness indices $\alpha_0$ and $\alpha_1$ which might be useful in applications.
Moreover, the authors of \cite{FJ90} needed to apply duality results in their proof (to handle the term which corresponds to $\M^+$ in Step 2 of our proof given below). 
This is not necessary in our case, but it would be possible of course.
Finally, due to the mild assumptions on the two (possibly different) multiscale grids $\nabla^i$, $i\in\{0,1\}$, in sharp contrast to \cite{FJ90}, our theorem is not restricted to spaces related to function spaces on the whole of $\R^d$.
Indeed, in \autoref{sect:change}, we will employ \autoref{thm:ad} to derive a result for Besov-type function spaces on (bounded) manifolds or domains, respectively.

\begin{proof}[Proof (of \autoref{thm:ad})]
Let $0<q\leq\infty$. 
Following the lines of \cite{FJ90}, we split the proof into three parts corresponding to different parameter constellations. 
To keep the presentation as streamlined as possible we moreover postpone some technicalities to the appendix in \autoref{sect:aux}.

\emph{Step 1 (case $\alpha_0=\alpha_1=0$ and $1<p<\infty$).}
For $p>1$ let $\M\in\ad_{p}^{0}$ and $\bm{a}\in b^{0}_{p,q}(\nabla^0)$. 
Writing $\M=\M^{-}+\M^{+}$, where we set
\begin{equation*}
	(\M^{-}\bm{a})_{(j,\xi)} 
	:= \sum_{0\leq k<j} \sum_{\eta \in \nabla^0_k} m_{(j,\xi),(k,\eta)} \, a_{(k,\eta)}
	\quad \text{and} \quad 
	(\M^{+}\bm{a})_{(j,\xi)} 
	:= \sum_{k\geq j} \sum_{\eta \in \nabla^0_k} m_{(j,\xi),(k,\eta)} \, a_{(k,\eta)}
\end{equation*}
for every $(j,\xi)\in\nabla^1$, we have to show that the associated linear operators $M^{-}$ and $M^+$ are bounded mappings from $b^{0}_{p,q}(\nabla^0)$ into itself, i.e., that
\begin{equation}\label{cond}
	\norm{ M^{\pm}\bm{a} \sep b^{0}_{p,q}(\nabla^1) } 
	\leq c \norm{\bm{a} \sep b^{0}_{p,q}(\nabla^0)}
\end{equation}
with some $c>0$ independent of $\bm{a}$.

Let us first consider $\M^-$ and $M^-$, respectively. Since $1<p<\infty$, the triangle inequality together with Minkowski's inequality yields for every fixed $j\in\N_0$,
\begin{align*}
	\left[ \sum_{\xi\in\nabla_j^1} \abs{ (\M^{-}\bm{a})_{(j,\xi)} }^p \right]^{1/p}
	&\leq \left[ \sum_{\xi\in\nabla_j^1} \left( \sum_{0\leq k < j} \abs{ \sum_{\eta \in \nabla^0_k} m_{(j,\xi),(k,\eta)} \, a_{(k,\eta)}} \right)^p \right]^{1/p} \\
	&\leq \sum_{0\leq k < j} \left( \sum_{\xi\in\nabla_j^1} \abs{ \sum_{\eta \in \nabla^0_k} \abs{m_{(j,\xi),(k,\eta)}} \, \abs{a_{(k,\eta)}} }^p \right)^{1/p}.
\end{align*}
Due to \autoref{def:ad} we have $\abs{m_{(j,\xi),(k,\eta)}} \leq C \cdot \omega_{(j,\xi),(k,\eta)}(\epsilon)$ for some $C=C(\epsilon)>0$ and all $(j,\xi)\in\nabla^1$, $(k,\eta)\in \nabla^0$. 
Note that $\sigma_p=0$, since $1<p<\infty$. Thus, we conclude
\begin{align}\label{boundM-}
	\left[ \sum_{\xi\in\nabla_j^1} \abs{ (\M^{-}\bm{a})_{(j,\xi)} }^p \right]^{1/p}
	&\lesssim \sum_{0\leq k < j} 2^{(k-j)(d/2+\epsilon)} \left( \sum_{\xi\in\nabla_j^1} \abs{ \sum_{\eta \in \nabla^0_k} \frac{\abs{a_{(k,\eta)}}}{\left[ 1+ 2^k \dist{\xi}{\eta} \right]^{d+\epsilon}} }^p \right)^{1/p}
\end{align}
for every $j\in\N_0$. 
Observe that for every fixed $I\in\N_0$ the sets $\nabla^i_I$, $i\in\{0,1\}$, equipped with the counting measure $\mu_I$ form $\sigma$-finite measure spaces and that $L_p(\nabla^i_I,\mu_I) = \ell_p(\nabla^i_I)$.
Hence, for every $j,k\in\N_0$ with  $0\leq k<j$, we can rewrite the sum in the brackets as
\begin{equation*}
	\norm{ T_{j,k,\epsilon}^- \left(\abs{a_{(k,\eta)}}\right)_{\eta\in\nabla^0_k} \sep \ell_p(\nabla_j^1) },
\end{equation*}
where $T_{j,k,\epsilon}^- \colon \ell_p(\nabla^0_k)\nach \ell_p(\nabla_j^1)$ is an integral (summation) operator with kernel
\begin{equation*}
	K_{j,k,\epsilon}^-(\xi,\eta) 
	:= \frac{1}{\left[ 1+ 2^k \dist{\xi}{\eta} \right]^{d+\epsilon}},
	\qquad \xi \in \nabla_j^1, \, \eta\in\nabla^0_k.
\end{equation*}
Using a Schur-type argument (see \autoref{lem:int_op}) together with \autoref{lem:sum} from \autoref{sect:aux} below, for $s=d+\epsilon>d$ its operator norm can be bounded by
\begin{align}
	\norm{ T_{j,k,\epsilon}^- \sep \mathcal{L}\left(\ell_p(\nabla^0_k),\ell_p(\nabla_j^1)\right)} 
	&\leq \left( \sup_{\eta\in\nabla^0_k} \sum_{\xi\in\nabla_j^1} \abs{K_{j,k,\epsilon}^-(\xi,\eta)} \right)^{1/p} \left( \sup_{\xi\in\nabla_j^1} \sum_{\eta\in\nabla^0_k} \abs{K_{j,k,\epsilon}^-(\xi,\eta)} \right)^{1/p'} \nonumber\\
	&\leq C' \left( \max\!\left\{1, 2^{(j-k)(d+\epsilon)}\right\} \right)^{1/p} \left( \max\!\left\{1, 2^{(k-k)(d+\epsilon)}\right\} \right)^{1/p'} \nonumber\\
	&= C' \, 2^{(j-k)(d+\epsilon)/p}, \label{norm_T-}
\end{align}
where $1/p'+1/p=1$ and $C'>0$ does not depend on $j$ and $k$.
Hence, from \link{boundM-} it follows
\begin{align}
	\left[ \sum_{\xi\in\nabla_j^1} \abs{ (\M^{-}\bm{a})_{(j,\xi)} }^p \right]^{1/p}
	&\lesssim \sum_{0\leq k < j} 2^{(k-j)(d/2+\epsilon)} 2^{(j-k)(d+\epsilon)/p} \left( \sum_{\eta\in\nabla^0_k} \abs{a_{(k,\eta)}}^p \right)^{1/p} \nonumber\\
	&= 2^{-jd\left[\frac{1}{2}-\frac{1}{p}\right]} \sum_{0\leq k < j} 2^{-(j-k)\epsilon/p'} \left[ 2^{kd\left[\frac{1}{2}-\frac{1}{p}\right]} \left( \sum_{\eta\in\nabla^0_k} \abs{a_{(k,\eta)}}^p \right)^{1/p} \right].\label{boundM-2}
\end{align}
Finally, we multiply by $2^{jd\left[\frac{1}{2}-\frac{1}{p}\right]}$, take the $\ell_q$-(quasi-)norm with respect to $j\in\N_0$, and apply \autoref{lem:rych} (with $\delta:=\epsilon/p'>0$ and $r:=1$) to obtain
\begin{equation*}
	\left\{ \sum_{j\in\N_0} 2^{jd\left[\frac{1}{2}-\frac{1}{p}\right]q} \left[ \sum_{\xi\in\nabla_j^1} \abs{ (\M^{-}\bm{a})_{(j,\xi)} }^p \right]^{q/p}\right\}^{1/q} 
	\lesssim \left\{ \sum_{k\in\N_0} 2^{k d\left[\frac{1}{2}-\frac{1}{p}\right]q} \left[ \sum_{\eta\in\nabla^0_k} \abs{ a_{(k,\eta)} }^p \right]^{q/p}\right\}^{1/q}, 
\end{equation*}
if $q<\infty$, and 
\begin{equation*}
	\sup_{j\in\N_0} 2^{jd\left[\frac{1}{2}-\frac{1}{p}\right]} \left[ \sum_{\xi\in\nabla_j^1} \abs{ (\M^{-}\bm{a})_{(j,\xi)} }^p \right]^{1/p} 
	\lesssim \sup_{k\in\N_0} 2^{k d\left[\frac{1}{2}-\frac{1}{p}\right]} \left[ \sum_{\eta\in\nabla^0_k} \abs{ a_{(k,\eta)} }^p \right]^{1/p},
\end{equation*}
if $q=\infty$, respectively.
Hence, we have shown \link{cond} for $M^-$.

We turn to $\M^+$ and $M^+$, respectively. The analogue of \link{boundM-} for fixed $j\in\N_0$ reads
\begin{equation*}
	\left[ \sum_{\xi\in\nabla_j^1} \abs{ (\M^{+}\bm{a})_{(j,\xi)} }^p \right]^{1/p}
	\lesssim \sum_{k \geq j} 2^{(j-k)(d/2+\epsilon)} \left( \sum_{\xi\in\nabla_j^1} \abs{ \sum_{\eta \in \nabla^0_k} \frac{\abs{a_{(k,\eta)}}}{\left[ 1+ 2^j \dist{\xi}{\eta} \right]^{d+\epsilon}} }^p \right)^{1/p}
\end{equation*}
such that the kernel of the associated integral operator $T^+_{j,k,\epsilon}\colon \ell_p(\nabla^0_k)\nach\ell_p(\nabla_j^1)$ is given by
\begin{equation*}
	K_{j,k,\epsilon}^+(\xi,\eta) 
	:= \frac{1}{\left[ 1+ 2^j \dist{\xi}{\eta} \right]^{d+\epsilon}},
	\qquad \xi \in \nabla_j^1, \, \eta\in\nabla^0_k.
\end{equation*}
Hence, \link{norm_T-} is replaced by $\norm{ T_{j,k,\epsilon}^+ } \leq C' \, 2^{(k-j)(d+\epsilon)/p'}$
such that \link{boundM-2} now reads
\begin{align*}
	\left[ \sum_{\xi\in\nabla_j^1} \abs{ (\M^{+}\bm{a})_{(j,\xi)} }^p \right]^{1/p}
	&\lesssim \, 2^{-jd\left[\frac{1}{2}-\frac{1}{p}\right]} \sum_{k \geq j} 2^{(j-k)\epsilon/p} \left[ 2^{kd\left[\frac{1}{2}-\frac{1}{p}\right]} \left( \sum_{\eta\in\nabla^0_k} \abs{a_{(k,\eta)}}^p \right)^{1/p} \right].
\end{align*}
Since $\delta:=\epsilon/p>0$, the assertion thus follows as before. This shows \link{cond} also for $M^+$ and hence it completes Step 1.

\emph{Step 2 (case $\alpha_0=\alpha_1=0$ and $0<p\leq 1$).}
Let $\bm{a}=(a_{(k,\eta)})_{(k,\eta)\in\nabla^0}\in b^0_{p,q}(\nabla^0)$ and 
choose $0 < r < p \leq 1$, i.e., $1<p/r<\infty$.
For every such $r$ we define $\tilde{\bm{a}}:=\tilde{\bm{a}}(r):= \left( \tilde{a}_{(k,\eta)} \right)_{(k,\eta)\in\nabla^0}$ by
\begin{equation*}
	\tilde{a}_{(k,\eta)}
	:=  2^{-k d\left[ \frac{1}{2} - \frac{r}{2} \right] } \abs{a_{(k,\eta)}}^r, \quad (k,\eta)\in\nabla^0,
	\quad \text{ so that } \quad
	\norm{\bm{a} \sep b^0_{p,q}(\nabla^0)} 
	= \norm{\tilde{\bm{a}} \sep b^0_{p/r,q/r}(\nabla^0)}^{1/r},
\end{equation*}
i.e., $\tilde{\bm{a}} \in b^0_{p/r,q/r}(\nabla^0)$.
Similarly, given a matrix $\M=\{m_{(j,\xi),(k,\eta)}\}_{(j,\xi)\in\nabla^1,(k,\eta)\in\nabla^0}$, we set
\begin{equation*}
	\tilde{\M}
	:=\tilde{\M}(r)
	:= \left\{ \tilde{m}_{(j,\xi),(k,\eta)} \right\}_{(j,\xi)\in\nabla^1,(k,\eta)\in\nabla^0}
	:= \left\{ 2^{(k-j)d\left[\frac{1}{2}-\frac{r}{2}\right]} \abs{m_{(j,\xi),(k,\eta)}}^r \right\}_{(j,\xi)\in\nabla^1,(k,\eta)\in\nabla^0}.
\end{equation*}
If we assume that there exists $\epsilon>0$ such that $\M$ belongs to $\ad_p^{0}$ then
straightforward calculations show that $\tilde{\M}=\tilde{\M}(r)\in \ad_{p/r}^{0}$ with $\tilde{\epsilon}:=\epsilon r + d(r/p-1)$.
Note that $\tilde{\epsilon}>0$, provided that we restrict ourselves to $r$ with $pd/(\epsilon p + d) < r < p$ (which can be done w.l.o.g.). 
Using Jensen's inequality we obtain
\begin{align*}
	\abs{\left( \M\bm{a} \right)_{(j,\xi)}}^p
	\leq \left( \sum_{(k,\eta)\in\nabla^0} \abs{m_{(j,\xi),(k,\eta)} a_{(k,\eta)}}^r \right)^{p/r}
	= 2^{jd\left[\frac{1}{2}-\frac{r}{2}\right]p/r} \abs{ \left( \tilde{\M}\tilde{\bm{a}} \right)_{(j,\xi)} }^{p/r}
\end{align*}
for every $(j,\xi)\in\nabla^1$. 
Therefore, for $q<\infty$, the associated operator $M$ satisfies
\begin{align*}
	\norm{M\bm{a} \sep b^0_{p,q}(\nabla^1)}^r
	&= \left( \sum_{j\in\N_0} 2^{jd\left[\frac{1}{2}-\frac{1}{p}\right]q} \left[ \sum_{\xi\in\nabla_j^1} \abs{\left(\M \bm{a} \right)_{(j,\xi)}}^p \right]^{q/p} \right)^{r/q} \\
	&\leq \left( \sum_{j\in\N_0} 2^{jd\left[\frac{1}{2}-\frac{1}{p/r}\right]q/r} \left[ \sum_{\xi\in\nabla_j^1} \abs{\left(\tilde{\M} \tilde{\bm{a}} \right)_{(j,\xi)}}^{p/r} \right]^{q/p} \right)^{q/r}
\end{align*}
which can be bounded from above by 
\begin{equation*}
	\norm{\tilde{M}\tilde{\bm{a}} \sep b^0_{p/r,q/r}(\nabla^1)} 
	\lesssim \norm{\tilde{\bm{a}} \sep b^0_{p/r,q/r}(\nabla^0)}
	= \norm{\bm{a} \sep b^0_{p,q}(\nabla^0)}^r.
\end{equation*}	
Here the last estimate follows from Step 1, since $\tilde{\M} \in\ad^{0}_{p/r}$ and $1<p/r<\infty$.
Clearly, the same is true also for $q=\infty$. 
Thus, we have shown that $\M\in\ad^{0}_{p}$ implies the continuity of $M\colon b^0_{p,q}(\nabla^0)\nach b^0_{p,q}(\nabla^1)$, as claimed.

\emph{Step 3 (case $\alpha_i\neq 0$).}
Following \cite{FJ90} we note that the result for the case $\alpha_i \neq 0$, $i\in\{0,1\}$, can be reduced to the assertion for $\alpha_0=\alpha_1=0$ as follows: 
Obviously, for $i\in\{0,1\}$, we have
\begin{equation*}
	\bm{a}=\left( a_{(k,\eta)} \right)_{(k,\eta)\in\nabla^i}\in b^{\alpha_i}_{p,q}(\nabla^i) 
	\qquad \text{if and only if} \qquad 
	\tilde{\bm{a}} := \left( 2^{k\alpha_i} a_{(k,\eta)} \right)_{(k,\eta)\in\nabla^i} \in b^{0}_{p,q}(\nabla^i)
\end{equation*}
with $\norm{\bm{a} \sep b^{\alpha_i}_{p,q}(\nabla^i)}=\norm{\tilde{\bm{a}} \sep b^{0}_{p,q}(\nabla^i)}$. Moreover,
$\M=\{m_{(j,\xi),(k,\eta)}\}_{(j,\xi)\in\nabla^1,(k,\eta)\in\nabla^0} \in \ad_{p}^{\alpha_0,\alpha_1}$ if and only if
$\tilde{\M}:=\{2^{j\alpha_1-k\alpha_0} m_{(j,\xi),(k,\eta)}\}_{(j,\xi)\in\nabla^1,(k,\eta)\in\nabla^0}$ belongs to $ \ad_{p}^{0}$. 
Since $\norm{M\bm{a} \sep b^{\alpha_1}_{p,q}(\nabla^1)}$ clearly equals $\norm{\tilde{M}\tilde{\bm{a}} \sep b^0_{p,q}(\nabla^1)}$, the linear operator $M\colon b^{\alpha_0}_{p,q}(\nabla^0) \nach b^{\alpha_1}_{p,q}(\nabla^1)$ with matrix $\M$ is bounded if and only if $\tilde{M}\colon b^{0}_{p,q}(\nabla^0) \nach b^{0}_{p,q}(\nabla^1)$ with matrix $\tilde{\M}$ is continuous.
As this argument holds for every $p$ and $q$, the proof is complete.
\end{proof}

\section{Besov-type spaces based on wavelet expansions}
\label{sect:Fun}
In this section we turn to function spaces. 
In particular, here we are going to extend our notion of Besov-type spaces established in \cite{DahWei2013} to a fairly general setting. 
These (quasi-)Banach spaces are subsets of the space of all square-integrable functions defined on some domain or manifold $\Gamma$. 
In view of the applications we have in mind, we are especially interested in bounded manifolds which admit a decomposition into smooth parametric images of the unit cube in $d$ spatial dimensions, since 
domains of this type are widely used in practice, e.g., in Computer Aided Geometric Design (CAGD). 
Moreover, as explained in the introduction, they are well-suited for the efficient numerical treatment of operator equations using FEM or BEM schemes based on multiscale analysis techniques. 
Consequently, in what follows we will focus on biorthogonal wavelet systems $\Psi$ (as they were constructed and analyzed, e.g., in \cite{CTU99,CTU00,DS99,HS06}) for such patchwise smooth manifolds $\Gamma$. 
In the Besov-type spaces $B^\alpha_{\Psi,q}(L_p(\Gamma))$ we then collect all those $L_2(\Gamma)$-functions, whose sequence of expansion coefficients w.r.t.\ some fixed basis $\Psi$ decays sufficiently fast (i.e., belongs to the space $b^{\alpha}_{p,q}(\nabla)$ introduced in \autoref{def:b}). 
Recently, it has been demonstrated that these Besov-type function spaces naturally arise in the analysis of adaptive numerical methods for operator equations on manifolds; the smoothness of the solutions (measured in these scales) determines the rate of their best $n$-term (wavelet) approximation which in turn serves as a benchmark for the performance of ideal adaptive algorithms.
For details we refer to \cite{DahWei2013}, as well as to the introduction of the current paper, and to the references given therein.

In \autoref{sect:domains} below we describe the setting for the domains or manifolds under consideration in detail. Afterwards we recall some fundamental ideas from the field of multiscale analysis and review basic features of the three special wavelet constructions on manifolds we are going to deal with.
Finally, \autoref{sect:Besov} is concerned with the definition of Besov-type spaces based on wavelet expansions, as well as with some of their theoretical properties which are relevant for practical applications.

\subsection{Domain decomposition and representation of geometry}\label{sect:domains}
When it comes to applications such as, e.g., the numerical treatment of integral equations defined on complicated geometries, often the following setting is assumed.

Given natural numbers $m$ and $d$ with $d \leq m$, let $\Gamma$ denote a bounded $d$-dimensional manifold in
$\R^m$ with or without a boundary.
We assume that $\Gamma$ is at least globally Lipschitz continuous and admits a decomposition
\begin{equation}\label{eq:manifold}
	\overline{\Gamma} = \bigcup_{i=1}^N \overline{\Gamma_i}
\end{equation}
into finitely-many, essentially disjoint \emph{patches} $\Gamma_i$,
i.e., $\Gamma_i \cap \Gamma_j = \leer$ for all $i\neq j$.
In addition, we assume that these patches are given as smooth parametric images of the $d$-dimensional unit cube which will serve as a reference domain. That is, we assume
\begin{equation*}
	\overline{\Gamma_i} = \kappa_i([0,1]^d),
\end{equation*}
where for each $i=1,\ldots,N$ the function $\kappa_i \colon \R^d \nach \R^m$ is supposed to be sufficiently regular.
Moreover, the splitting of $\Gamma$ needs to be conforming in the sense that for all $i\neq j$ the pullback of the intersection $\overline{\Gamma_i}\cap \overline{\Gamma_j}$ is either empty or a lower dimensional face of $[0,1]^d$.
In the latter case the set $\overline{\Gamma_i}\cap \overline{\Gamma_j}$ is called \emph{interface} between the patches $\Gamma_i$ and $\Gamma_j$.
Finally, we assume that the parametrizations $\kappa_i$ are chosen in a way that for every interface there exists a permutation $\pi_{i,j}$ such that 
\begin{equation*}
	\kappa_j \circ \pi_{i,j} \circ \kappa_i^{-1} = \mathrm{Id}
	\qquad \text{on} \qquad \overline{\Gamma_i} \cap \overline{\Gamma_j},
\end{equation*}
where $\mathrm{Id}$ denotes the identical mapping and $\pi_{i,j}(\bm{x}):=(x_{\pi_{i,j}(1)}, \ldots, x_{\pi_{i,j}(d)})$ for $\bm{x}=(x_1,\ldots,x_d)$ in $[0,1]^d$.

For the remainder of this paper a domain or manifold which meets all these requirements is said to be \emph{decomposable} or \emph{patchwise smooth}. 

\begin{example}
Practically relevant examples for the manifolds under consideration are given by surfaces of bounded polyhedra in three spatial dimensions, i.e., $\Gamma=\partial\Omega$ with $\Omega\subset\R^3$. 
The reader may think of, e.g., Fichera's corner $\Omega=[-1,1]^3\setminus [0,1]^3$ which often serves as a model domain for numerical simulations. (Here the reentrant corner causes a singularity in the solutions to large classes of operator equations that is typical for problems on non-smooth domains.)
\hfill$\square$
\end{example}

\begin{remark}
We stress the fact that although our setting is tailored to handle boundary integral equations defined on (two-dimensional) closed surfaces, it covers open manifolds and bounded domains (of arbitrary dimension) as well.
Thus, in principle the approach given here is suitable also for the treatment of boundary value problems involving partial differential operators.
\hfill$\square$
\end{remark}

\subsection{Multiresolution analysis and biorthogonal wavelets on patchwise smooth manifolds}
\label{sect:MRA_wavelets}
One powerful tool to construct approximate solutions to operator equations defined on decomposable domains or manifolds in the sense the previous section is given by (adaptive) wavelet methods. 
In this approach the equation under consideration is discretized using a suitable set of basis functions stemming from a multiresolution analysis (a precise definition is given below). 
Then truncated versions of the resulting infinite linear system are solved. 
The attractive features of wavelets (such as smoothness, cancellation, and support properties which together imply the needed compression) combined with adaptive refinement and coarsening strategies finally yield an efficient algorithm.
However, the construction of wavelet bases on patchwise smooth manifolds $\Gamma$ is far away from being trivial as we shall now explain.

Let us assume for a moment that $\Gamma$ denotes an arbitrary set (equipped with some metric) which additionally allows the definition of $L_2(\Gamma)$, the space of (equivalence classes of) square-integrable functions $f\colon\Gamma\nach\C$, equipped with some inner product $\distr{\cdot}{\cdot}$. 
Moreover, we assume to be given a \emph{multiresolution analysis} (MRA) for this space, i.e., a sequence $\V=(V_j)_{j\in\N}$ of closed linear subspaces of $L_2(\Gamma)$ which satisfies
\begin{equation*}
	V_j \subset V_{j+1}, 
	\quad j\in\N,
	\qquad \text{and} \qquad
	\overline{\bigcup_{j\in\N} V_j}^{\,L_2(\Gamma)} = L_2(\Gamma).
\end{equation*}
Now the main idea in multiscale analysis is to find a suitable \emph{system of wavelet type} $\Psi^\Gamma := \bigcup_{j\in \N_0} \Psi^\Gamma_j \subset L_2(\Gamma)$ such that the functions at \emph{level} $j$ span some complement of $V_j$ in $V_{j+1}$. 
Given any countable subset $X \subset L_2(\Gamma)$ we let $S(X)$ denote $\overline{\spann{X}}^{L_2(\Gamma)}$. Then
we assume that
\begin{equation*}
	V_1 = S(\Psi_0^\Gamma)
	\qquad \text{and} \qquad 
	V_{j+1} = V_j \oplus S(\Psi_j^\Gamma), 
	\quad j \in\N,
\end{equation*}
where, for every $j\in\N_0$, $\Psi^\Gamma_j:=\{ \psi^\Gamma_{j,\xi} \sep \xi \in \nabla_j^\Psi\}$ is indexed by some set $\nabla_j^\Psi$ such that the sequence $\nabla^\Psi:=(\nabla_j^\Psi)_{j\in\N_0}$ 
forms a multiscale grid for $\Gamma$ in the sense of \autoref{ass:struct}.
Of course it would be favorable if $\Psi^\Gamma$ would constitute an orthonormal basis for $L_2(\Gamma)$, but, in practice, such bases are not always feasible.
However, as we shall see, there exist \emph{biorthogonal} constructions which retain most of the desired properties of orthonormal bases. 

In the (more flexible) biorthogonal setting a second multiresolution analysis $\tilde{\V}=(\tilde{V}_j)_{j\in\N}$ of $L_2(\Gamma)$, together with a corresponding system of wavelet type $\tilde{\Psi}^\Gamma$ (again indexed by $\nabla^\Psi$), is needed such that the following duality w.r.t.\ the inner product $\distr{\cdot}{\cdot}$ holds for all $j\in\N$:
\begin{equation*}
	\tilde{V}_j \, \bot \, S(\Psi_j^\Gamma)
	\qquad \text{and} \qquad
	V_j \, \bot \, S(\tilde{\Psi}_j^\Gamma).
\end{equation*}
It then follows that both the systems $\Psi^\Gamma$ and $\tilde{\Psi}^\Gamma$ form (Schauder-) bases of $L_2(\Gamma)$ and that they are biorthogonal in the sense that
\begin{equation*}
	\distr{\psi^\Gamma_{j,\xi}}{\tilde{\psi}^\Gamma_{k,\eta}}
	= \begin{cases}
		1, & \text{if } j=k \text{ and } \xi=\eta,\\
		0, & \text{otherwise.}
	\end{cases}
\end{equation*}
Finally, under suitable conditions, we can assume that $\Psi^\Gamma$ and $\tilde{\Psi}^\Gamma$ even form biorthogonal \emph{Riesz bases} for $L_2(\Gamma)$, i.e., every $u\in L_2(\Gamma)$ can be written as
\begin{equation}\label{eq:expansion}
	u 
	= \sum_{(j,\xi)\in\nabla^\Psi} \distr{u}{\tilde{\psi}^\Gamma_{j,\xi}} \psi^\Gamma_{j,\xi} 
	= \sum_{(j,\xi)\in\nabla^\Psi} \distr{u}{\psi^\Gamma_{j,\xi}} \tilde{\psi}^\Gamma_{j,\xi} 
\end{equation}
and we have the norm equivalences
\begin{equation}\label{eq:norm_equiv}
	\norm{u \sep L_2(\Gamma)} 
	\sim \norm{ \left( \distr{u}{\tilde{\psi}^\Gamma_{j,\xi}} \right)_{(j,\xi)\in\nabla^\Psi} \sep \ell_2(\nabla^\Psi)}
	\sim \norm{ \left( \distr{u}{\psi^\Gamma_{j,\xi}} \right)_{(j,\xi)\in\nabla^\Psi} \sep \ell_2(\nabla^\Psi)}.
\end{equation}
Consequently, then all (primal and dual) wavelets are normalized in the sense that
\begin{equation}\label{eq:normalized}
	\norm{ \psi^\Gamma_{j,\xi} \sep L_2(\Gamma)}
	\sim \norm{\tilde{\psi}^\Gamma_{j,\xi} \sep L_2(\Gamma)} 
	\sim 1,
	\qquad (j,\xi)\in\nabla^\Psi.
\end{equation}
For the rest of this paper we shall use $\Psi$ as a shortcut for (bi-)orthogonal wavelet Riesz bases $(\Psi^\Gamma,\tilde{\Psi}^\Gamma)$ on the set $\Gamma$ with the properties just mentioned.

\begin{remark}
When dealing with the (more restrictive) orthogonal setting we simply assume that the primal and the dual MRA, as well as the corresponding systems of wavelet type, coincide, i.e., then $\tilde{\V}=\V$ and $\tilde{\Psi}^\Gamma=\Psi^\Gamma$.
\hfill$\square$
\end{remark}

Although these concepts of multiscale representations can be employed in a quite general framework, for the ease of presentation and in view of the applications we have in mind, we are mainly interested in sets $\Gamma$ which meet the requirements of \autoref{sect:domains} in what follows. 
In particular, we only consider manifolds of finite diameter and explicitly exclude the case of unbounded domains in $\R^d$.
For this setting a suitable inner product for $L_2(\Gamma)$, which is equivalent to the canonical one, is given by
\begin{equation}\label{eq:inner_prod}
	\distr{f}{g} := \sum_{i=1}^N \distr{f\circ\kappa_i}{g\circ\kappa_i}_{L_2([0,1]^d)},
	\qquad f,g\in L_2(\Gamma),
\end{equation}
because it allows to shift the challenging problem of constructing wavelets from the (possibly complicated) manifold $\Gamma$ to the unit cube $[0,1]^d$.
Thanks to the tensor product structure of this reference domain, multivariate wavelets then can be easily deduced from  univariate ones which in turn are constructed with the help of some dual pair $(\theta, \tilde{\theta})$ of refinable functions on the real line.
When required by the final application even (homogeneous) boundary conditions can be incorporated at this point; see, e.g., \cite{DahSchn1998, Pri2010}.

An important family of underlying dual pairs is based on B-splines, as they allow very efficient point evaluation and quadrature routines:
\begin{equation*}
	\left(\theta, \tilde{\theta}\right)
	 := \left({_{D}}{\theta}, {_{D,\tilde{D}}}{\tilde{\theta}}\right),
	\quad \text{where} \quad D,\tilde{D}\in\N \quad \text{with} \quad \tilde{D}\geq D \quad \text{and} \quad D+\tilde{D} \quad \text{even}.
\end{equation*}
Therein ${_{D}}{\theta}$ denotes the $D$th-order centered cardinal B-spline and ${_{D,\tilde{D}}}{\tilde{\theta}}$ is some compactly supported, refinable function which is exact of order $\tilde{D}$. 
Moreover, it can be checked that the regularity of ${_{D}}{\theta}$ equals $D-1/2$ and ${_{D,\tilde{D}}}{\tilde{\theta}}$ can be chosen in a way such that its regularity increases proportional with $\tilde{D}$, i.e.,
\begin{equation*}
	\gamma 
	:= \sup\!\left\{s>0 \sep {_{D}}{\theta} \in H^s(\R)\right\} = D-\frac{1}{2}
	\quad \text{and} \quad 
	\tilde{\gamma} 
	:= \sup\!\left\{s>0 \sep {_{D,\tilde{D}}}{\tilde{\theta}} \in H^s(\R)\right\} \sim \tilde{D}.
\end{equation*}

By now there exist several constructions that use the idea of lifting wavelets from the cube to the patches $\Gamma_i$ of the manifold under consideration. 
In the sequel we particularly focus on the prominent case of \emph{composite wavelet bases} which were initially established by Dahmen and Schneider in \cite{DS99} and further developed by Harbrecht and Stevenson \cite{HS06}. Another important set of wavelets is due to Canuto, Tabacco and Urban \cite{CTU99,CTU00}.
These three constructions (which will be labeled by $\Psi_{\mathrm{DS}}$, $\Psi_{\mathrm{HS}}$, and $\Psi_{\mathrm{CTU}}$, respectively) mainly differ in the treatment of wavelets supported in the vicinity of the interfaces. 
Without going into details, we mention that here some ``gluing'' or ``matching'' procedure is necessary in order to finally obtain wavelets which are continuous across the interfaces and, at the same time, yield the compression properties required by applications.

For these three systems it has been shown that for all choices of construction parameters the resulting wavelets provide norm equivalences 
\begin{equation}\label{eq:norm_equivalence_sob}
	\norm{ u \sep H^s(\Gamma)} 
	\sim \left( \sum_{j=0}^\infty 2^{2js} \sum_{\xi\in \nabla_j^\Psi} \abs{\distr{u}{\tilde{\psi}_{j,\xi}^\Gamma}}^2 \right)^{1/2} 
\end{equation}
for the scale of classical Sobolev spaces $H^s(\Gamma)$ in some (very limited) range of smoothness parameters $s$ which does not depend on $D$ and $\tilde{D}$; see, e.g., \cite[Thm.~4.6.1]{DS99}. 
Moreover, for $s\in(-\tilde{\gamma},D-1/2)$ the same equivalences hold for generalized spaces $H_s(\Gamma)$ based on $\distr{\cdot}{\cdot}$ which coincide with $H^s(\Gamma)$ provided that $\abs{s}$ is sufficiently small. 
This shows that these (Hilbert) spaces are characterized by \emph{all} bases $\Psi=(\Psi^\Gamma,\tilde{\Psi}^\Gamma)$ under consideration as long as their construction parameters $D^{\Psi}$ and $\tilde{D}^{\Psi}$ are chosen sufficiently large.
In \autoref{sect:change} we are going to extend this assertion to a fairly large class of (quasi-)Banach spaces: We show that, in the sense of equivalent (quasi-)norms, any two wavelet systems $\Psi,\Phi\in\{\Psi_{\mathrm{DS}}, \Psi_{\mathrm{HS}}, \Psi_{\mathrm{CTU}}\}$ generate the same Besov-type spaces (see \autoref{def:Bspq} below)
\begin{equation*}
	B^{\alpha}_{\Psi,q}(L_p(\Gamma))=B^{\alpha}_{\Phi,q}(L_p(\Gamma))
\end{equation*}
provided that the smoothness of the space $\alpha$ is smaller than some quantity depending on $D^{\Psi}$ and $D^{\Phi}\in\N$. 
To prove this we will have to bound the inner products $\distr{\psi_{k,\eta}^{\Gamma}}{\tilde{\phi}^{\Gamma}_{j,\xi}}$ subject to the relation of $(j,\xi)\in\nabla^{\Phi}$ and $(k,\eta)\in\nabla^{\Psi}$ to each other.
For this purpose, the following properties shared by all the three bases of interest will be useful.
As their proof is quite technical, we postpone it to the appendix; see \autoref{sect:proofs}.

\begin{lemma}\label{lem:cancel}
	For a decomposable $d$-dim.\ manifold $\Gamma$ let 
	$\Psi=(\Psi^\Gamma,\tilde{\Psi}^\Gamma) \in \{\Psi_{\mathrm{DS}}, \Psi_{\mathrm{HS}}, \Psi_{\mathrm{CTU}}\}$ 
	denote a wavelet basis (as constructed in \cite{DS99}, \cite{HS06}, or \cite{CTU99,CTU00}) indexed by some multiscale grid $\nabla^\Psi=(\nabla_j^\Psi)_{j\in\N_0}$ for $\Gamma$.
	Then for all $j\in\N_0$ and each $\xi=(y,t)\in\nabla_j^\Psi \subset \Gamma\times\T$ we have that
		\begin{enumerate}[label=(P\arabic{*}), ref=P\arabic{*}]
			\item \label{P1} $y \in \supp{\psi^{\Gamma}_{j,\xi}} \cap \supp{\tilde{\psi}^{\Gamma}_{j,\xi}}$,
			\item \label{P2} $\diam\!\left(\supp{\psi^{\Gamma}_{j,\xi}}\right) \sim \diam\!\left(\supp{\tilde{\psi}^{\Gamma}_{j,\xi}}\right) \sim 2^{-j}$, and
			\item \label{P3} there exist cubes $\tilde{C}_{j,\xi}^{\,i} \subset [0,1]^d$ with $\abs{\tilde{C}_{j,\xi}^{\,i}} \lesssim 2^{-j d}$ and $\kappa_i^{-1}\!\left(\supp \tilde{\psi}_{j,\xi}^\Gamma \cap \Gamma_i \right) \subseteq \tilde{C}_{j,\xi}^{\,i}$ such that for every $s\in(d/2, D^{\Psi}]$ and all functions $f\colon \Gamma\nach\C$ it holds 
					\begin{equation}\label{patchwise_bound}
						\abs{\distr{f}{\tilde{\psi}_{j,\xi}^\Gamma}}
						\lesssim \sum_{i=1}^N 2^{-js} \abs{f\circ\kappa_i}_{H^s(\tilde{C}^{\,i}_{j,\xi})},
					\end{equation}
					provided that the right-hand side is finite. 
					Moreover, a completely analogous statement holds true for cubes $C_{j,\xi}^{\, i}$ 
					when $\tilde{\psi}_{j,\xi}^\Gamma$ and $D^{\Psi}$ are replaced by 
					$\psi_{j,\xi}^\Gamma$ and $\tilde{D}^{\Psi}$, respectively.
		\end{enumerate}
\end{lemma}	
Since particularly the property \link{P3} will be essential in the following, let us add some comments on it: 
Roughly speaking, the estimate \link{patchwise_bound} says that the expansion coefficients of functions $f$ on $\Gamma$ which correspond to wavelets $\psi_{j,\xi}^\Gamma$ supported on more than one patch~$\Gamma_i$ are bounded by the \emph{patchwise} (Sobolev) regularity of the pullbacks of $f$ to the unit cube. 
The lower bound $d/2$ for $s$ is due to the fact that the wavelet constructions under consideration incorporate terms which involve function evaluations at the interfaces, whereas the upper bound $D^{\Psi}$ is implied by the degree of polynomial exactness of the underlying scaling functions.

\subsection{Definition and properties of Besov-type function spaces}\label{sect:Besov}
Besov spaces essentially generalize the concept of Sobolev spaces. On $\R^d$ they are typically
defined using harmonic analysis, finite differences, moduli of smoothness, or interpolation. Characteristics such as embeddings, interpolation results, or approximation properties
of these scales then require deep proofs within the classical theory of function spaces.
Often they are obtained by reducing the assertion of interest to the level of sequences spaces
by means of characterizations in terms of building blocks (atoms, local means, quarks, or
wavelets). To mention at least a few references the interested reader is referred to the monographs
\cite{RS96, T06}, as well as to the articles \cite{DJP92, FJ90, KMM07}.

As outlined in \cite{DahWei2013}, the definition of Besov spaces on manifolds deserves some care: When following the usual approach based on local charts the smoothness of the spaces then is limited by the global regularity of the underlying manifold.
On the other hand, the theoretical analysis of adaptive algorithms for problems defined on (patchwise smooth) manifolds naturally requires higher-order smoothness spaces of Besov type.
Therefore, in \cite[Def.~4.1]{DahWei2013} we introduced a notion of Besov-type spaces $B_{\Psi,q}^\alpha(L_p(\Gamma))$ on specific two-dimensional closed manifolds such as boundaries of certain polyhedral domains $\Omega \subset\R^3$.
The definition was based on expansions w.r.t.\ special wavelets bases $\Psi$ that satisfy a number of additional properties; see \cite{DahWei2013} for details. 
Now we are going to extend this definition to a much more general setting: Besides enlarging the range of admissible parameters, here we significantly weaken the assumptions on the underlying set $\Gamma$, as well as on the wavelet bases used in the construction of the spaces.
As illustrated by \autoref{ex:besov_spaces} below, the subsequent definition thus covers a wide variety of function spaces on bounded or unbounded, smooth or non-smooth domains and manifolds of arbitrary dimension.

\begin{defi}\label{def:Bspq}
	Let $\Psi=(\Psi^\Gamma,\tilde{\Psi}^\Gamma)$ denote any (bi-)orthogonal wavelet Riesz basis for $L_2(\Gamma)$ w.r.t.\ some inner product $\distr{\cdot}{\cdot}$ which is indexed by a multiscale grid $\nabla^{\Psi}=\left(\nabla_j^\Psi \right)_{j\in\N_0}$ of dimension $d\in\N$ for $\Gamma$. Then
	\begin{itemize}
	 	\item[(i)] the tuple $(\alpha,p,q)$ is said to be \emph{admissible} if 
	 	\begin{equation*}
			0 < p \begin{cases}
	 				< \infty & \quad \text{when } \nabla^{\Psi} \text{ satisfies } \link{A4a},\\
	 				\leq 2 & \quad \text{when } \nabla^{\Psi} \text{ satisfies } \link{A4b},
	 		\end{cases}
	 	\end{equation*}
	and if one of the following conditions applies:
	\begin{itemize}
		\item[$\bullet$)] $\alpha > d \cdot \max\!\left\{0, \frac{1}{p} - \frac{1}{2} \right\}$ \quad and \quad $0 < q \leq \infty$,
		\item[$\bullet$)] $\alpha = d \cdot \max\!\left\{0, \frac{1}{p} - \frac{1}{2} \right\}$ \quad and \quad $0 < q \leq 2$.
	\end{itemize}
\item[(ii)]
	for any admissible parameter tuple $(\alpha,p,q)$ let $B_{\Psi,q}^\alpha(L_p(\Gamma))$ denote the collection of all complex-valued functions $u\in L_2(\Gamma)$ such that the (quasi-)norm defined by
	\begin{equation*}
		\norm{u \sep B_{\Psi,q}^\alpha(L_p(\Gamma))}
		:= \norm{ \left( \distr{u}{\tilde{\psi}^{\Gamma}_{j,\xi}} \right)_{(j,\xi)\in\nabla^{\Psi}} \sep b^\alpha_{p,q}\!\left(\nabla^{\Psi}\right)}
	\end{equation*}
	is finite. Therein the sequence space $b^\alpha_{p,q}(\nabla^{\Psi})$ is defined as in \autoref{def:b}.
	\end{itemize}
\end{defi}

Some comments are in order.

\begin{remark}\label{rem:B}
Observe that due to our assumptions every $u\in L_2(\Gamma)$ admits a unique expansion w.r.t.\ the primal wavelet system $\Psi^\Gamma$, where the corresponding sequence of coefficients belongs to $\ell_2(\nabla^{\Psi})=b_{2,2}^0(\nabla^{\Psi})$; cf.\ \link{eq:expansion} and \link{eq:norm_equiv}.
On the other hand, \autoref{prop:seq-embedding} implies that $b^\alpha_{p,q}(\nabla^{\Psi}) \hookrightarrow \ell_2(\nabla^{\Psi})$ for all admissible parameter tuples. Therefore every function with finite $B_{\Psi,q}^\alpha(L_p(\Gamma))$-quasi-norm belongs to $L_2(\Gamma)$. In fact, we have $B_{\Psi,q}^\alpha(L_p(\Gamma)) \hookrightarrow L_2(\Gamma)$; also compare \autoref{prop:standard_emb} below.
\hfill$\square$
\end{remark}

\begin{example}\label{ex:besov_spaces}
Let us illustrate the flexibility of \autoref{def:Bspq} by means of the some examples:
\begin{itemize}
	\item[(i)] Most importantly, \autoref{def:Bspq} covers spaces on $d$-dimensional manifolds which are patchwise smooth in the sense of \autoref{sect:domains}.
As exposed in \autoref{sect:MRA_wavelets}, in this setting biorthogonality of functions on $\Gamma$ usually is realized w.r.t.\ the patchwise inner product~\link{eq:inner_prod}. Suitable wavelet systems are given by $\Psi=(\Psi^\Gamma,\tilde{\Psi}^\Gamma) \in \{\Psi_{\mathrm{DS}}, \Psi_{\mathrm{HS}}, \Psi_{\mathrm{CTU}}\}$, as constructed in \cite{DS99}, \cite{HS06}, and \cite{CTU99,CTU00}, respectively.
Note that then no restriction on $p$ is imposed as decomposable domains \link{eq:manifold} are assumed to be bounded; cf.\ \link{A4a}.
	\item[(ii)] Assume that $\Gamma$ denotes a compact $d$-dimensional $\mathcal{C}^\infty$ manifold. 
		Then Besov spaces $B^\alpha_{p,q}(\Gamma)$ of arbitrary smoothness are well-defined by lifting spaces of distributions on $\R^d$ to $\Gamma$ using local charts together with an (overlapping) resolution of unity; see, e.g., \cite[Def.~5.1]{T08}. 
		Without going into details, we state that (for the range of admissible parameter tuples) these spaces coincide with our spaces $B_{\Psi,q}^\alpha(L_p(\Gamma))$ introduced in \autoref{def:Bspq}, provided that the wavelet system under consideration satisfies additional requirements; cf.\ \cite[Prop.~5.32]{T08}. 
		For an elaborate discussion we refer to \cite[Ch.~5]{T08}.
	\item[(iii)] Finally, note that also classical Besov function spaces $B^\alpha_{p,q}(\R^d)$ are covered. For this purpose, we may take a system of Daubechies wavelets which forms an orthogonal basis w.r.t.\ the canonical inner product on $L_2(\R^d)$. The coincidence of our \autoref{def:Bspq} with the usual definition based on Fourier techniques then is shown, e.g., in \cite[Thm.~1.64]{T06}; see also the proof of \autoref{prop:seq-embedding} in \autoref{sect:proofs}.
	\hfill$\square$
\end{itemize}
\end{example}

In the remainder of this section we briefly describe a couple of properties satisfied by the scale of function spaces just introduced which yield attractive implications for practical applications, e.g., in the context of regularity studies of operator equations. 
For details we again refer to \cite{DahWei2013}.
To begin with, we note that (formally) the spaces constructed in \autoref{def:Bspq} depend on the concrete choice of the wavelet basis $\Psi$. 
As already mentioned, in \autoref{sect:change} below we will show that under quite natural conditions all wavelet systems under consideration actually lead to the same Besov-type spaces.

From the properties of the sequence spaces $b^\alpha_{p,q}(\nabla^{\Psi})$ it immediately follows that all spaces $B_{\Psi,q}^\alpha(L_p(\Gamma))$ are quasi-Banach spaces; cf.\ \autoref{rem:props_b}. Moreover, they are Banach spaces if and only if $\min\{p,q\}\geq 1$ and Hilbert spaces if and only if $p=q=2$.
In fact, for small smoothness parameters $\alpha=s \in [0,\min\{3/2, s_{\Gamma}\})$ and $p=q=2$ our Besov-type spaces coincide with the classical Sobolev Hilbert spaces $H^s(\Gamma)$ (in the sense of equivalent norms), 
provided that a suitable wavelet basis is used; see, e.g., \cite[Thm.~4.6.1]{DS99}, or \cite[Cor.~5.7]{CTU99}.
Here the number $s_\Gamma$ is related to the smoothness of the underlying manifold $\Gamma$.
The coincidence then simply follows from the fact that the right-hand side of \link{eq:norm_equivalence_sob} equals $\norm{( \langle u, \tilde{\psi}^{\Gamma}_{j,\xi}\rangle )_{(j,\xi)\in\nabla^{\Psi}} \sep b^s_{2,2}\!\left(\nabla^{\Psi}\right)}$ which in turn defines the norm of $u$ in $B_{\Psi,2}^s(L_2(\Gamma))$.

Furthermore, \autoref{prop:seq-embedding} (together with \autoref{rem:B}) implies the subsequent characterization of embeddings between Besov-type spaces which is listed here for the sake of completeness.

\begin{corollary}[Standard embeddings]\label{prop:standard_emb}
Choose $\Psi$, $\nabla^{\Psi}$, and $\Gamma$ as in \autoref{def:Bspq}. Moreover, for $\alpha,\gamma \in \R$, let $(\alpha+\gamma,p_0,q_0)$ and $(\alpha,p_1,q_1)$ denote admissible parameter tuples.
If $\Gamma$ is bounded (i.e., $\nabla^{\Psi}$ satisfies \link{A4a}) then we have the continuous embedding
\begin{equation*}
	B_{\Psi,q_0}^{\alpha+\gamma}(L_{p_0}(\Gamma)) \hookrightarrow B_{\Psi,q_1}^{\alpha}(L_{p_1}(\Gamma))
\end{equation*}
if and only if one of the following conditions applies:
\begin{itemize}
	\item[$\bullet$)] $\gamma > d \cdot \max\!\left\{0,\frac{1}{p_0} - \frac{1}{p_1}\right\}$,
	\item[$\bullet$)] $\gamma = d \cdot \max\!\left\{0,\frac{1}{p_0} - \frac{1}{p_1}\right\}$ \quad and \quad $q_0\leq q_1$.
\end{itemize}
Furthermore, if $\Gamma$ is unbounded (i.e., $\nabla^{\Psi}$ satisfies \link{A4b} instead of \link{A4a}) then a corresponding characterization holds true with the additional condition $p_0\leq p_1$.
\end{corollary}

The embeddings stated in \autoref{prop:standard_emb} can be illustrated by \emph{DeVore-Triebel diagrams}; see \autoref{fig:DeVoreTriebel}. Therein the solid lines, starting from the point $(1/2; 0)$ which corresponds
to the space $L_2(\Gamma)=B^0_{\Psi,2}(L_2(\Gamma))$, describe the boundaries of the respective areas of admissible parameters; cf.\ \autoref{def:Bspq}(i). 
In both cases these areas are limited at the right-hand side by the so-called \emph{adaptivity scale of Besov-type spaces} $B^{\alpha_\tau}_{\Psi,\tau}(L_\tau(\Gamma))$, where $\tau$ and $\alpha_\tau$ are related via \link{eq:adapt}.
Moreover, the shaded regions refer to all spaces which are embedded into $H^s(\Gamma)=B^s_{\Psi,2}(L_2(\Gamma))$, whereas the arrows indicate limiting cases for possible embeddings of the space $B_{\Psi,q_0}^{\alpha+\gamma}(L_{p_0}(\Gamma))$. Restrictions imposed by the fine indices $q$ are not visualized.
\begin{figure}[h]
	\begin{center}
\begin{tikzpicture}
\draw [gray, thick] (1,3) -- (3.5,5.5);
\shade [left color=gray!75, right color=gray!10, shading angle=165] (1,3)--(1,5.5)--(3.5,5.5)--(1,3);
\shade [left color=gray!75, right color=gray!10, shading angle=195] (1,3)--(1,5.5)--(-1,5.5)--(-1,3)--(1,3);

\draw [->,>=stealth'] (-1,-0.2) -- (-1,5.7) node[above] {$\alpha$};
\draw [->,>=stealth'] (-1.2,0) -- (5,0) node[right] {$\frac{1}{p}$};

\draw [thick] (1,0) -- (5,4);
\draw [very thick] (-1,0) -- (1,0);
\draw [dotted] (1,0) -- (1,5.5);
\fill (1,0) circle (2pt) node[anchor=south east] {$L_2$};

\draw (-1.3,3) node {$s$};
\draw [gray] (-1.0,3) -- (1,3);
\draw [dashed] (-1.05,3) -- (-0.95,3);
\draw [dashed] (1,0) -- (1,3);
\draw (1,0) -- (1,-0.05);
\draw (1,-0.5) node {$\frac{1}{2}$};
\fill (1,3) circle (2pt) node[anchor=north east] {$H^{s}$};
\draw (-1.05,4.5) node[anchor=east] {$s_{\Gamma}$} -- (-0.95,4.5);

\draw (-1.3,3.5) node[anchor=east, left=-0.2] {$\alpha+\gamma$};
\draw [dashed] (-1.05,3.5) -- (3,3.5);
\fill (3,3.5) circle (2pt) node[anchor=south, above right=0.05 and -0.45] {$B^{\alpha+\gamma}_{\Psi,q_0}(L_{p_0})$};
\draw [dashed] (3,-0.05) -- (3,0.7);
\draw [dashed] (3,1.4) -- (3,3.5);
\draw (3,-0.5) node {$\frac{1}{p_0}$};

\draw [->,>=stealth'] (3,3.5) -- (4,3.5);
\draw [dotted] (4,3.5)--(4.5,3.5);
\draw [->,>=stealth'] (3,3.5) -- (2,2.5);
\draw [dotted] (2,2.5)--(-0.5,0);

\draw (-1.3,1) node {$\alpha_\tau$};
\draw [dashed] (-1.05,1.1) -- (2.1,1.1);
\fill (2.1,1.1) circle (2pt) node[anchor=west, right=0.1] {$B^{\alpha_\tau}_{\Psi,\tau}(L_\tau)$};
\draw [dashed] (2.1,-0.05) -- (2.1,1.1);
\draw (2.1,-0.5) node {$\frac{1}{\tau}$};
\end{tikzpicture}
		\hfill
\begin{tikzpicture}
\draw [->,>=stealth'] (0,-0.2) -- (0,5.7) node[above] {$\alpha$};
\draw [->,>=stealth'] (-0.2,0) -- (5,0) node[right] {$\frac{1}{p}$};

\draw [gray, thick] (1,3) -- (3.5,5.5);
\shade [left color=gray!75, right color=gray!10, shading angle=165] (1,3)--(1,5.5)--(3.5,5.5)--(1,3);

\draw [thick] (1,0) -- (5,4);
\draw [very thick] (1,0) -- (1,5.5);
\fill (1,0) circle (2pt) node[anchor=south east] {$L_2$};

\draw (-0.3,3) node {$s$};
\draw [dashed] (-0.05,3) -- (1,3);
\draw (1,0) -- (1,-0.05);
\draw (1,-0.5) node {$\frac{1}{2}$};
\fill (1,3) circle (2pt) node[anchor=north east] {$H^{s}$};
\draw (-0.05,4.5) node[anchor=east] {$s_\Gamma$} -- (0.05,4.5);

\draw (-0.3,3.5) node[anchor=east, left=-0.2] {$\alpha+\gamma$};
\draw [dashed] (-0.05,3.5) -- (3,3.5);
\fill (3,3.5) circle (2pt) node[anchor=south, above right=0.05 and -0.45] {$B^{\alpha+\gamma}_{\Psi,q_0}(L_{p_0})$};
\draw [dashed] (3,-0.05) -- (3,0.7);
\draw [dashed] (3,1.4) -- (3,2);
\draw (3,-0.5) node {$\frac{1}{p_0}$};

\draw [->,>=stealth'] (3,3.5) -- (3,2.5);
\draw [dotted] (3,2.5)--(3,2);
\draw [->,>=stealth'] (3,3.5) -- (2,2.5);
\draw [dotted] (2,2.5)--(1,1.5);

\draw (-0.3,1) node {$\alpha_\tau$};
\draw [dashed] (-0.05,1.1) -- (2.1,1.1);
\fill (2.1,1.1) circle (2pt) node[anchor=west, right=0.1] {$B^{\alpha_{\tau}}_{\Psi,\tau}(L_\tau)$};
\draw [dashed] (2.1,-0.05) -- (2.1,1.1);
\draw (2.1,-0.5) node {$\frac{1}{\tau}$};
\end{tikzpicture}
	\end{center}
\caption{Embeddings for Besov-type spaces $B^\alpha_{\Psi,q}(L_p(\Gamma))$ on bounded (left) and on unbounded (right) domains or manifolds $\Gamma$, resp.; cf. \autoref{prop:standard_emb}.}\label{fig:DeVoreTriebel}
\end{figure}
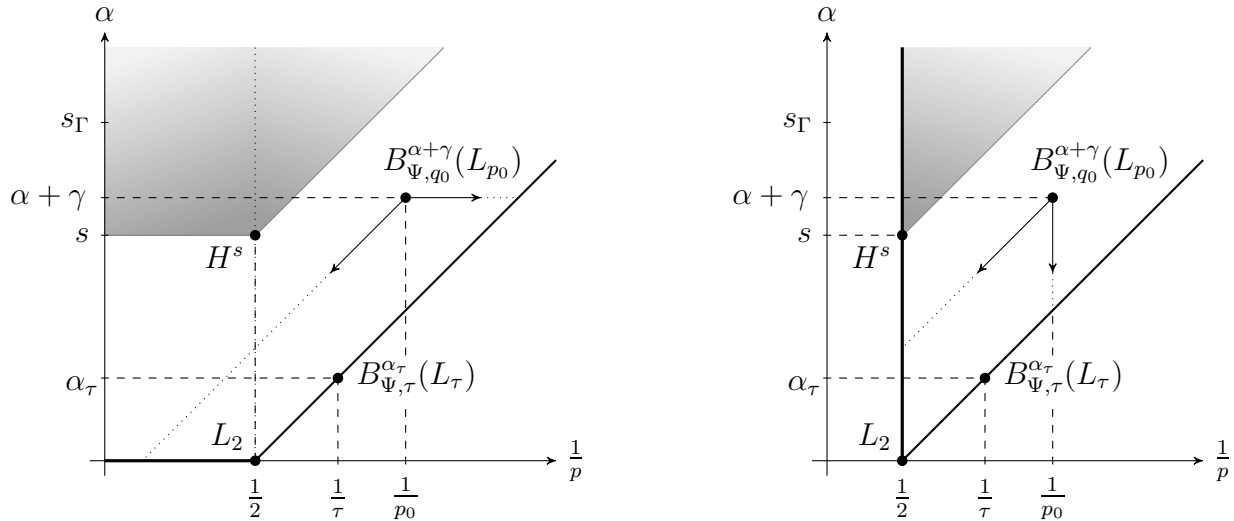

When it comes to applications on bounded domains or manifolds, approximation properties such as best $n$-term rates are of particular interest. 
Without going into details, let us recall that roughly speaking the numbers $\sigma_n(F;\B,G)$, $n\in\N_0$, describe the minimal error of approximating the embedding $F\hookrightarrow G$ by means of finite linear combinations of elements from the dictionary $\B$. 
For an exact definition we refer to \cite[Sect.~4.2]{DahWei2013}. 
There also a proof (based on results shown in \cite{DNS06,HanSic2011}) of the next proposition for $d=2$ can be found. The arguments easily carry over to the general case discussed here.
\begin{prop}[Best $n$-term approximation on bounded manifolds]\label{prop:nterm}
	Choose $\Psi$, $\nabla^{\Psi}$, and $\Gamma$ as in \autoref{def:Bspq} and assume $\Gamma$ to be bounded, i.e., suppose that $\nabla^\Psi$ satisfies \link{A4a}. Moreover, for $\alpha,\gamma \in \R$, let $(\alpha+\gamma,p_0,q_0)$ and $(\alpha,p_1,q_1)$ denote admissible parameter tuples. 
	Then
\begin{itemize}
	\item[$\bullet$)] $\gamma > d \cdot \max\!\left\{0,\frac{1}{p_0} - \frac{1}{p_1}\right\}$ implies
	\begin{equation*}
		\sigma_n \!\left( B_{\Psi,q_0}^{\alpha+\gamma}(L_{p_0}(\Gamma));\Psi^{\Gamma},B_{\Psi,q_1}^\alpha(L_{p_1}(\Gamma)) \right) 
		\sim n^{-\gamma/d},
	\end{equation*}
\item[$\bullet$)] $\gamma = d \cdot \max\!\left\{0,\frac{1}{p_0} - \frac{1}{p_1}\right\}$ and $q_0 \leq q_1$ implies
	\begin{equation*}
		\sigma_n \!\left( B_{\Psi,q_0}^{\alpha+\gamma}(L_{p_0}(\Gamma));\Psi^{\Gamma},B_{\Psi,q_1}^\alpha(L_{p_1}(\Gamma)) \right) 
		\sim n^{-\min\{\gamma/d,\, 1/q_0-1/q_1\}}. 
	\end{equation*}
	\end{itemize}	
\end{prop}
\begin{remark}
Note that (using results stated in \cite{HanSic2011}) a corresponding characterization (but with quite different rates of convergence!) can be derived easily also for spaces on unbounded sets~$\Gamma$. 
As for applications bounded domains or manifolds, respectively, are much more important, we will not follow this line of research here.
Indeed, as exposed already in the introduction, the quantity $\sigma_n \!\left( B_{\Psi,\tau}^{\alpha_\tau}(L_{\tau}(\Gamma));\Psi^{\Gamma}, B_{\Psi,2}^0(L_{2}(\Gamma)) \right)$ with $(\alpha_\tau,\tau)$ as in \link{eq:adapt} serves as a benchmark for the performance of (ideal) adaptive algorithms that use at most $n$ wavelets from the dictionary $\Psi^{\Gamma}$ and provide an approximation in the norm of  $L_2(\Gamma)=B_{\Psi,2}^0(L_{2}(\Gamma))$.
The reason is that, on the one hand, due to \autoref{prop:nterm}, the best $n$-term approximation rates linearly depend on the difference in smoothness and, on the other hand, the spaces $B_{\Psi,\tau}^{\alpha_\tau}(L_{\tau}(\Gamma))$ from the adaptivity scale \link{eq:adapt} provide the weakest norms among all (Besov-type) spaces of fixed regularity which are contained in $L_2(\Gamma)$; cf.\ \autoref{prop:standard_emb}.
\hfill$\square$
\end{remark}

Finally, besides many other interesting properties which are typical for classical Besov spaces (e.g., defined via harmonic analysis), our scale of Besov-type spaces $B_{\Psi,q}^\alpha(L_p(\Gamma))$ satisfies the following interpolation assertions w.r.t\ the \emph{real} and the (extended) \emph{complex method} which we denote by $(\cdot,\cdot)_{\Theta,q}$ and $[\cdot,\cdot]_{\Theta}$, respectively. 
For a comprehensive treatment of interpolation of (quasi-)Banach spaces we refer to \cite{BL76, KMM07, T83} and to the references therein.
\begin{prop}[Interpolation]
	Choose $\Psi$, $\nabla^{\Psi}$, and $\Gamma$ as in \autoref{def:Bspq} and let $(\alpha_0,p_0,q_0)$ and $(\alpha_1,p_1,q_1)$ denote admissible parameter tuples.
	For $0<\Theta<1$ we set
	\begin{equation*}
		s_{\Theta} 
		:= (1-\Theta) \, \alpha_0 + \Theta \, \alpha_1, \qquad 
		\frac{1}{p_\Theta} 
		:= \frac{1-\Theta}{p_0} + \frac{\Theta}{p_1}, \qquad \text{and} \qquad
		\frac{1}{q_\Theta} 
		:= \frac{1-\Theta}{q_0} + \frac{\Theta}{q_1}.
	\end{equation*}
	\begin{itemize}
		\item[$\bullet$)] If $\alpha_0\neq \alpha_1$ and $p=p_0=p_1$ then for all $0<q\leq\infty$ and every $0<\Theta<1$ we have
			\begin{equation*}
				\Big( B_{\Psi,q_0}^{\alpha_0}(L_{p}(\Gamma)),\, B_{\Psi,q_1}^{\alpha_1}(L_{p}(\Gamma)) \Big)_{\Theta,q} 
				= B_{\Psi,q}^{s_\Theta}(L_{p}(\Gamma)).
			\end{equation*}
		\item[$\bullet$)] If $\min\{q_0,q_1\}<\infty$ then for all $0<\Theta<1$ it holds
			\begin{equation*}
				\Big[ B_{\Psi,q_0}^{\alpha_0}(L_{p_0}(\Gamma)), \, B_{\Psi,q_1}^{\alpha_1}(L_{p_1}(\Gamma)) \Big]_{\Theta} 
				= B_{\Psi,q_\Theta}^{s_\Theta}(L_{p_\Theta}(\Gamma)).
			\end{equation*}
	\end{itemize}
\end{prop}
\begin{proof}
As it has been shown in \cite[Prop.~4.5]{DahWei2013} (for the special case $d=2$), interpolation results for Besov-type spaces $B_{\Psi,q}^\alpha(L_p(\Gamma))$ can be reduced to corresponding assertions for sequence spaces which in turn follow from interpolation properties of (classical) Besov spaces $B^\alpha_{p,q}(\R^d)$ defined on the whole of $\R^d$. 
This type of arguments does not depend on the dimension and can be applied for all methods that fulfill the so-called interpolation property; cf.\ \cite[Rem.~6.3]{DahWei2013}. 
Thus, in our case it suffices to refer to \cite[Thm.~2.4.2(i)]{T83} for the real method and to \cite[Thm.~9.1]{KMM07} for the (extended) complex method, respectively.
\end{proof}

\section{Change of basis embeddings for Besov-type spaces}
\label{sect:change}
As outlined above, our \autoref{def:Bspq} of Besov-type spaces $B_{\Psi,q}^\alpha(L_p(\Gamma))$ formally depends on the concrete choice of the wavelet basis $\Psi$ and its construction parameters $D^{\Psi}$ and $\tilde{D}^{\Psi}$, respectively.
In order to find conditions which imply that different bases $\Psi$ and $\Phi$ generate the same Besov-type space $B_{\Psi,q}^\alpha(L_p(\Gamma)) = B_{\Phi,q}^\alpha(L_p(\Gamma))$ in the sense of equivalent (quasi-)norms, we now employ the theory of almost diagonal matrices developed in \autoref{sect:Seq} to investigate properties which yield corresponding one-sided \emph{change of basis embeddings}.

Note that, in general, different constructions of wavelet bases might accomplish (bi-) orthogonality w.r.t.\ different inner products. 
Indeed, depending on the desired properties we like to assemble, on patchwise smooth manifolds $\Gamma$, say, it is reasonable to construct wavelets which are biorthogonal with respect to $\distr{\cdot}{\cdot}$ as defined in \link{eq:inner_prod}, or w.r.t.\ the canonical scalar product $\distrmod{\cdot}{\cdot}$ on $L_2(\Gamma)$.
The following proposition addresses this issue, as it is stated in a quite general form.

\begin{prop}\label{prop:localization}
	Let $\Psi$ and $\Phi$ denote two wavelet Riesz bases for $L_2(\Gamma)$ which are (bi-) orthogonal w.r.t.\ the inner products $\distr{\cdot}{\cdot}$ and $\distrmod{\cdot}{\cdot}$, respectively. 
	Moreover, suppose that these bases are indexed by multiscale grids $\nabla^{\Psi}$ and $\nabla^\Phi$ 
	for $\Gamma$, respectively, and assume that for some admissible parameter tuple $(\alpha,p,q)$ the associated Gramian matrix satisfies
	\begin{equation}\label{cond:Gram}
		\M_{\Psi\nach\Phi} 
		:= \{m_{(j,\xi),(k,\eta)}\}_{(j,\xi)\in\nabla^{\Phi},(k,\eta)\in\nabla^{\Psi}}
		= \left\{ \distrmod{\psi_{k,\eta}^{\Gamma}}{\tilde{\phi}^{\Gamma}_{j,\xi} }  \right\}_{(j,\xi)\in\nabla^{\Phi},(k,\eta)\in\nabla^{\Psi}} \in\ad_{p}^{\alpha}.
	\end{equation}
	Then $B_{\Psi,q}^\alpha(L_p(\Gamma)) \hookrightarrow B_{\Phi,q}^\alpha(L_p(\Gamma))$.
\end{prop}

\begin{proof}
We essentially follow the lines of the proof of \cite[Thm.~3.7]{FJ90}.
By definition, every $u\in B_{\Psi,q}^\alpha(L_p(\Gamma))$ can be expanded into
\begin{equation*}
	u 
	= \sum_{k\in\N_0} \sum_{\eta\in\nabla_k^\Psi} a_{(k,\eta)} \, \psi_{k,\eta}^{\Gamma}
	\quad \text{with} \quad
	\bm{a} 
	:= \left( a_{(k,\eta)} \right)_{(k,\eta)\in\nabla^\Psi} 
	= \left( \distr{u}{\tilde{\psi}_{k,\eta}^{\Gamma}} \right)_{(k,\eta)\in\nabla^\Psi}
	\in b^\alpha_{p,q}(\nabla^\Psi).
\end{equation*}
Note that, since $B_{\Psi,q}^\alpha(L_p(\Gamma)) \hookrightarrow L_2(\Gamma)$ and $\Phi=( \Phi^{\Gamma}, \tilde{\Phi}^{\Gamma} )$ is a $\distrmod{\cdot}{\cdot}$-biorthogonal Riesz basis for $L_2(\Gamma)$, the sequence $\tilde{\bm{a}} := \left(\tilde{a}_{(j,\xi)}\right)_{(j,\xi)\in\nabla^\Phi} := \left( \distrmod{u}{\tilde{\phi}^{\Gamma}_{j,\xi}} \right)_{(j,\xi)\in\nabla^\Phi}$ is well-defined. 
Moreover, it holds $\tilde{\bm{a}} = M_{\Psi\nach\Phi} \bm{a}$. 
That is, for all $j\in\N_0$ and $\xi\in\nabla_j^\Phi$, we have
\begin{align*}
	\tilde{a}_{(j,\xi)}
	=  \distrmod{\sum_{k\in\N_0} \sum_{\eta\in\nabla_k^\Psi} a_{(k,\eta)} \, \psi_{k,\eta}^{\Gamma}}{\,\tilde{\phi}^{\Gamma}_{j,\xi}}
	= \sum_{(k,\eta)\in\nabla^\Psi} m_{(j,\xi),(k,\eta)} \, a_{(k,\eta)}
	= \left( \M_{\Psi\nach\Phi} \bm{a} \right)_{(j,\xi)}.
\end{align*}
Thus, from \autoref{thm:ad} it follows
\begin{align*}
	\norm{u \sep B_{\Phi,q}^\alpha(L_p(\Gamma))} 
	&= \norm{ \tilde{\bm{a}} \sep b^\alpha_{p,q}\!\left(\nabla^\Phi\right)}  \\
	&= \norm{ M_{\Psi\nach\Phi}  \bm{a} \sep b^\alpha_{p,q}\!\left(\nabla^\Phi\right)} \\
	&\lesssim \norm{ \bm{a} \sep b^\alpha_{p,q}\!\left(\nabla^\Psi\right)} \\
	&= \norm{u \sep B_{\Psi,q}^\alpha(L_p(\Gamma))}.
\end{align*}
This shows that the identity (mapping $B_{\Psi,q}^\alpha(L_p(\Gamma))$ into $B_{\Phi,q}^\alpha(L_p(\Gamma))$), induced by the operator $M_{\Psi\nach\Phi}\colon b^\alpha_{p,q}\!\left(\nabla^\Psi\right) \nach b^\alpha_{p,q}\!\left(\nabla^\Phi\right)$ which in turn is represented by the matrix $\M_{\Psi\nach\Phi}$ defined in \link{cond:Gram}, indeed is continuous, as claimed.
\end{proof}

Next let us apply the general concept presented in \autoref{prop:localization} above to the practically relevant case of Besov-type spaces generated by wavelet bases $\Psi, \Phi \in \{\Psi_{\mathrm{DS}},\Psi_{\mathrm{HS}}, \Psi_{\mathrm{CTU}}\}$ on patchwise smooth ($d$-dimensional) manifolds $\Gamma$ in the sense of \autoref{sect:domains}.
As described in \autoref{sect:MRA_wavelets}, all of these constructions are biorthogonal with respect to the same inner product~\link{eq:inner_prod} and all of them are built up from univariate centered cardinal B-splines of order $D^{\Psi}$ and $D^{\Phi}$ with regularity $\gamma^{\Psi}$ and $\gamma^{\Phi}$, respectively. 
Again the corresponding dual quantities are denoted by $\tilde{D}^{\Psi}$, etc.
Then the combination of \autoref{prop:localization} with \autoref{lem:cancel} implies the following result.

\begin{theorem}\label{thm:localized_composite}
	For a patchwise smooth $d$-dim.\ manifold $\Gamma$ let 
	$\Psi, \Phi \in \{\Psi_{\mathrm{DS}}, \Psi_{\mathrm{HS}}, \Psi_{\mathrm{CTU}}\}$ 
	denote two wavelet bases as constructed in \cite{DS99}, \cite{HS06}, or \cite{CTU99,CTU00}, respectively. 
	Moreover, assume that their construction parameters satisfy
\begin{equation*}
	\min\{D^{\Phi}, \tilde{\gamma}^{\Phi}, \tilde{D}^{\Psi}, \gamma^{\Psi}\}> d/2.
\end{equation*} 
Then for all admissible tuples of parameters $(\alpha,p,q)$ with
\begin{equation}\label{cond:maxalpha}
	0 \leq \alpha < \min\{D^{\Phi}, \gamma^{\Psi}\}
\end{equation} 
the continuous embedding $B_{\Psi,q}^\alpha(L_p(\Gamma)) \hookrightarrow B_{\Phi,q}^\alpha(L_p(\Gamma))$ holds true.
\end{theorem}
\begin{proof}
\emph{Step 1.}
In order to prove the claim we like to apply \autoref{prop:localization}. 
Thus we have to show that the Gramian matrix (w.r.t the change of basis from $\Psi$ to $\Phi$)
\begin{equation*}
		\M_{\Psi\nach\Phi} 
		:= \left\{ \distr{\psi_{k,\eta}^{\Gamma}}{\tilde{\phi}^{\Gamma}_{j,\xi} }  \right\}_{(j,\xi)\in\nabla^{\Phi},(k,\eta)\in\nabla^\Psi}
	\end{equation*}
belongs to the class $\ad_{p}^{\alpha}$ (cf.\ \autoref{def:ad}) for $\alpha$ and $p$ under consideration. 
Here $\nabla^\Psi$ and $\nabla^\Phi$ denote the associated multiscale grids of dimension $d$ for $\Gamma$ and $\distr{\cdot}{\cdot}$ is defined in $\link{eq:inner_prod}$.
Due to the monotonicity of the classes $\ad_{p}^{\alpha}$ (see \autoref{rem:ad}(ii)) it suffices to consider the limiting case $p=\tau=\tau(\alpha)$ with
\begin{equation}\label{eq:adapt_scale}
	\tau^{-1} := \frac{\alpha}{d} + \frac{1}{2}
\end{equation}
and $\alpha$ that satisfies \link{cond:maxalpha}.
Furthermore, as we will show in Step~2 below, it follows from the support conditions \link{P1} and \link{P2} in \autoref{lem:cancel} that $[1+ \min\!\left\{ 2^k, 2^j \right\} \dist{\xi}{\eta}] \sim 1$, so that it is enough to show that there exists $\epsilon > 0$ such that
\begin{equation}\label{cond:min}
	\abs{ \distr{\psi_{k,\eta}^{\Gamma}}{\tilde{\phi}^{\Gamma}_{j,\xi}} }
	\lesssim \min\!\left\{ 2^{-(j-k)(d/2+\alpha+\epsilon)}, 2^{(j-k)(d/2-\alpha+\epsilon+ \sigma_\tau)} \right\}
	\quad \text{for all} \quad (j,\xi)\in\nabla^{\Phi},(k,\eta)\in\nabla^\Psi.
\end{equation} Afterwards, in Step 3, we complete the proof by showing that \link{cond:min} is implied by \link{P3} in \autoref{lem:cancel}.

\emph{Step 2.}
For $\zeta:=(y,t)\in \Gamma\times\T$ and $r>0$ let $B(\zeta,r):=\{y' \in \Gamma \sep \rho_\Gamma(y',y) < r \}$ denote the open ball of radius~$r$ around~$y$ in $\Gamma$. Then \link{P1} and \link{P2} imply that
\begin{equation*}
	\supp{\psi^{\Gamma}_{k,\eta}} \cap \supp{\tilde{\phi}^{\Gamma}_{j,\xi}}	 \subseteq B(\eta,c'\, 2^{-k}) \cap B(\xi,c'\, 2^{-j})
\end{equation*}
for some $c'>0$, all $j,k\in\N_0$, and every $\xi\in\nabla_j^\Phi$, $\eta\in\nabla_k^\Psi$, respectively.
Note that the latter intersection is empty if $\dist{\xi}{\eta} > c'(2^{-k}+2^{-j})$ which in turn shows that 
$\distr{\psi^{\Gamma}_{k,\eta}}{\tilde{\phi}^{\Gamma}_{j,\xi}}\neq 0$ only if
\begin{equation*}
	1\leq 1 + \min\!\left\{ 2^k, 2^j \right\} \dist{\xi}{\eta} \leq c''
\end{equation*}
for some $c'' \geq 1$ which does not depend on $j$ and $k$. Therefore \link{cond_ad}, i.e., membership of $\M_{\Psi\nach\Phi}$ in $\ad_{p}^{\alpha}$, is equivalent to \link{cond:min}, as promised.

\emph{Step 3.}
We show \link{cond:min}. For this purpose, we note that
\begin{equation*}
	\sigma_\tau
	= d \cdot \max\!\left\{ \frac{1}{\tau}-1,0 \right\} 
	= \left\{\begin{array}{l}
   \left.
 		\begin{aligned}
	    	 0, 	 	 & \quad \text{if} \quad 1 \leq \tau \leq 2,\\
		\!\!	d/\tau - d, & \quad \text{if} \quad	0 < \tau < 1
   		\end{aligned}
   	\right\} 
   	= \begin{cases}
   		0, 			  & \text{if} \quad 0 \leq \alpha \leq d/2,\\
   		\alpha - d/2, & \text{if} \quad d/2 < \alpha,
   	\end{cases}
	\end{array} \right.
\end{equation*}
due to \link{eq:adapt_scale}.
This leads to the observation that
\begin{equation*}
	\frac{d}{2}-\alpha +\epsilon +\sigma_\tau
	= \left\{\begin{array}{l}
   \left.
 		\begin{aligned}
	    \!\!	 d/2-\alpha + \epsilon,  & \quad \text{if} \quad 0 \leq \alpha \leq d/2\\
				 \epsilon, & \quad \text{if} \quad	d/2 < \alpha
   		\end{aligned}
   	\right\}
   	\geq \epsilon > 0
	\end{array} \right.
\end{equation*}
such that the proof of \link{cond:min} naturally splits into the cases $j\geq k$ and $j<k$. 
For $j,k\in\N_0$ with $j\geq k$ we apply the first part of \link{P3} in \autoref{lem:cancel} for the basis $\Phi$ and $f:=\psi_{k,\eta}^\Gamma$ with $\eta\in\nabla_k^\Psi$. 
Observe that the patchwise regularity of this primal wavelet is as large as the smoothness of the underlying univariate spline used for its construction. Hence, given $i\in\{1,\ldots,N\}$, we conclude
\begin{equation*}
	\abs{\psi_{k,\eta}^\Gamma \circ \kappa_i}_{H^s(\tilde{C}_{j,\xi}^i)}
	\lesssim 2^{k s} \norm{\psi_{k,\eta}^\Gamma \circ \kappa_i \sep L_2(\tilde{C}_{j,\xi}^i)}
	\qquad \text{for all} \qquad 
	0\leq s < \gamma^{\Psi}.
\end{equation*}
For $s\in\N$ this simply follows from the multiscale structure of the wavelets. The case $s\notin\N$ can be derived using standard interpolation arguments.
Furthermore,
\begin{align*}
	\norm{\psi_{k,\eta}^\Gamma \circ \kappa_i \sep L_2(\tilde{C}_{j,\xi}^i)}^2
	= \int_{\tilde{C}_{j,\xi}^i} \abs{\psi_{k,\eta}^\Gamma (\kappa_i(x))}^2 \d x
	\leq \norm{\psi_{k,\eta}^\Gamma \sep L_\infty(\Gamma)}^2 \cdot \abs{\tilde{C}_{j,\xi}^i} \lesssim 2^{k d} \cdot 2^{-j d},
\end{align*}
since $\psi_{k,\eta}^\Gamma$ is $L_2(\Gamma)$-normalized; see \link{eq:normalized}. 
Combining the two last estimates with \link{patchwise_bound} thus gives
\begin{equation*}
	\abs{\distr{\psi_{k,\eta}^\Gamma}{\tilde{\phi}_{j,\xi}^\Gamma}} 
	\lesssim \sum_{i=1}^N 2^{-j s} \, 2^{k s}\, 2^{(k-j) d/2} 
	\sim 2^{-(j-k)(d/2+s)}
\end{equation*}
for all $d/2 < s \leq D^{\Phi}$ with $s<\gamma^{\Psi}$ and $\xi\in\nabla_j^\Phi$, $\eta\in\nabla_k^\Psi$ with $j\geq k$ in $\N_0$.
Note that, due to the assumption \link{cond:maxalpha}, we can find some $s$ in this range which is strictly larger than $\alpha$.
Choosing $\epsilon>0$ sufficiently small then yields $2^{-(j-k)(d/2+s)}\leq 2^{-(j-k)(d/2+\alpha+\epsilon)}$ which finally shows \link{cond:min} for $j\geq k$.

We are left with the case $j<k$. 
Using the second part of \link{P3} in \autoref{lem:cancel} (for the basis~$\Psi$, the index $\eta\in\nabla_k^\Psi$, and $f:=\tilde{\phi}_{j,\xi}^\Gamma$ with $\xi\in\nabla_j^\Phi$) together with the same arguments as before, we deduce the bound
\begin{equation*}
	\abs{\distr{\psi_{k,\eta}^\Gamma}{\tilde{\phi}_{j,\xi}^\Gamma}} 
	= \abs{\distr{\tilde{\phi}_{j,\xi}^\Gamma}{\psi_{k,\eta}^\Gamma}} \lesssim 2^{-(k-j)(d/2+s)}
\end{equation*}
for all $d/2<s\leq \tilde{D}^{\Psi}$ with $s<\tilde{\gamma}^{\Phi}$ and $j<k$ in $\N_0$.
Observe that for every such $s$ we have the estimate $2^{-(k-j)(d/2+s)}\leq 2^{(j-k)(d/2-\alpha+\epsilon+\sigma_\tau)}$, provided that $\epsilon>0$ is chosen small enough.
Therefore \link{cond:min} also holds for $j<k$ which completes the proof.
\end{proof}

As an immediate consequence of \autoref{thm:localized_composite} we conclude the main result of this paper. 
It states that all wavelet bases $\Psi, \Phi \in \{\Psi_{\mathrm{DS}},\Psi_{\mathrm{HS}}, \Psi_{\mathrm{CTU}}\}$ induce the same Besov-type spaces $B^\alpha_{\Psi,q}(L_p(\Gamma))=B^\alpha_{\Phi,q}(L_p(\Gamma))$ on patchwise smooth manifolds $\Gamma$ (in the sense of equivalent quasi-norms), provided that the primal sets of wavelets are of sufficiently large order of cancellation and regularity compared to the smoothness parameter $\alpha$ of the space.

\begin{theorem}\label{thm:equivalence}
	Given some $d$-dimensional manifold $\Gamma$ which is patchwise smooth in the sense of \autoref{sect:domains} let $\Psi=(\Psi^{\Gamma},\tilde{\Psi}^\Gamma)$ and $\Phi=(\Phi^{\Gamma},\tilde{\Phi}^\Gamma)$ denote two wavelet bases from 
	$\{\Psi_{\mathrm{DS}}, \Psi_{\mathrm{HS}}, \Psi_{\mathrm{CTU}}\}$ as constructed in \cite{DS99}, \cite{HS06}, or \cite{CTU99,CTU00}, respectively, and assume that their construction parameters satisfy
\begin{equation*}
	\min\!\left\{ D^{\Psi}, \tilde{D}^{\Psi}, \gamma^{\Psi}, \tilde{\gamma}^{\Psi}, D^{\Phi}, \tilde{D}^{\Phi}, \gamma^{\Phi}, \tilde{\gamma}^{\Phi} \right\} > d/2.
\end{equation*}
Then, for all admissible tuples of parameters $(\alpha,p,q)$ with
\begin{equation*}
	0 \leq \alpha < \min\!\left\{ D^{\Psi}, D^{\Phi}, \gamma^{\Psi}, \gamma^{\Phi} \right\},
\end{equation*}
it holds $B^\alpha_{\Psi,q}(L_p(\Gamma))=B^\alpha_{\Phi,q}(L_p(\Gamma))$ in the sense of equivalent (quasi-)norms.
\end{theorem}

\section{Appendix}\label{sect:appendix}
\subsection{Auxiliary assertions}\label{sect:aux}
This part of the appendix is concerned with auxiliary (technical) assertions that are needed in our proofs.
We start with a result which can be easily derived from \cite[Lem.~2]{Ry99}.
\begin{lemma}\label{lem:rych}
Let $0<r<\infty$, as well as $0<q\leq \infty$, and $\bm{x}:=\left(x_k\right)_{k\in\N_0} \in \ell_q(\N_0)$.
Then for all $\delta>0$ there exists a constant $c=c(\delta,r,q)>0$ such that both the quantities
\begin{align*}
	\norm{ \left( \left[ \sum_{k < j} 2^{-\delta(j-k) r} \abs{x_k}^r \right]^{1/r} \right)_{j\in\N_0} \sep \ell_q(\N_0)}
	\quad \text{and} \quad 
	\norm{ \left( \left[ \sum_{k \geq j} 2^{\delta(j-k) r} \abs{x_k}^r \right]^{1/r} \right)_{j\in\N_0} \sep \ell_q(\N_0)}
\end{align*}	
are upper bounded by $c \cdot \norm{\left( x_k \right)_{k\in\N_0} \sep \ell_q(\N_0)}$.
\end{lemma}

The next assertion constitutes a generalization of Schur's Lemma to the case of $\sigma$-finite measure spaces. 
It can be shown by straightforward calculations along the lines of \cite[(0.10)]{Fo95}.
\begin{lemma}\label{lem:int_op}
Let $(X,\mu)$ and $(Y,\nu)$ denote $\sigma$-finite measure spaces, $1\leq p \leq \infty$, as well as $1/p'+1/p=1$.
Moreover, assume that the measurable function $K\colon X \times Y \nach \C$ satisfies
\begin{align*}
	C_1(K) := \esssup_{x\in X} \int_Y \abs{K(x,y)} \, \d\nu(y) <\infty 
	\quad \text{and} \quad 
	C_2(K) := \esssup_{y\in Y} \int_X \abs{K(x,y)} \, \d\mu(x) <\infty.
\end{align*}
Then the integral operator $T\colon L_p(X,\mu) \nach L_p(Y,\nu)$, given by 
\begin{equation*}
	f \mapsto Tf := \int_X K(x,\cdot) f(x) \, \d\mu(x),
\end{equation*}
is well-defined and satisfies $\norm{T}:=\norm{T \sep \mathcal{L}(L_p(X,\mu),L_p(Y,\nu))} \leq C_1(K)^{1/p} \cdot C_2(K)^{1/p'}$.
\end{lemma}

Finally, in the proof of \autoref{thm:ad} we make use of the following estimate.
Therein $\Gamma$ denotes an arbitrary set furnished with some metric.
\begin{lemma}\label{lem:sum}
Let $\nabla=(\nabla_j)_{j\in\N_0}$ denote a multiscale grid of dimension $d\in\N$ for some set $\Gamma$ in the sense of \autoref{ass:struct}.
Then, for each $j,k\in\N_0$, every $s > d$ and all finite sets $\T\neq\emptyset$, we have
\begin{equation*}
	\sup_{x\in\Gamma\times \T} \sum_{\xi \in \nabla_j} \frac{1}{\left[ 1+2^k \dist{\xi}{x} \right]^s} \leq C \cdot \max\!\left\{1, 2^{(j-k)s}\right\}
\end{equation*}
with some $C>0$ which does not depend on $j$ and $k$.
\end{lemma}

\begin{proof}
Let $x\in\Gamma\times \T$ be fixed. Due to \link{A1} there exists $\xi_x \in \nabla_j$ such that $\dist{\xi_x}{x} \leq c_1 \, 2^{-j}$. 
Hence, the triangle inequality implies $\dist{\xi}{x} \geq \dist{\xi}{\xi_x} - c_1\,2^{-j}$, i.e.,
\begin{equation*}
	1 + 2^k \dist{\xi}{x} \geq 1 - c_1\, 2^{k-j} + 2^k \dist{\xi}{\xi_x}
	\quad \text{for all} \quad \xi \in \nabla_j.
\end{equation*}
Next we define layer sets $L := \left\{ \xi \in \nabla_j \sep 0 \leq \dist{\xi}{\xi_x} < (\ceil{c_1}+1)\, 2^{-j} \right\}$, and
\begin{align*}
	L_{\ell} := \left\{ \xi \in \nabla_j \sep \ell\, 2^{-j} \leq \dist{\xi}{\xi_x} < (\ell+1)\, 2^{-j} \right\}
	\quad \text{for all} \quad 
	\ell \in \N 
	\quad \text{with} \quad 
	\ell \geq \ceil{c_1}+1.
\end{align*}
Since $\nabla$ is uniformly well-separated (see \link{A2}), we note that $\# L$ is bounded (uniformly in $x$ and~$j$). The estimate $\# L_\ell \lesssim \ell^{d-1}$ holds because of \link{A3}.
Moreover, we obviously have the coincidence $\nabla_j = L \cup \bigcup_{\ell\geq \ceil{c_1}+1} L_\ell$.
Thus we may estimate
\begin{align*}
	\sum_{\xi \in \nabla_j} \frac{1}{\left[ 1+2^k \dist{\xi}{x} \right]^s}
	&= \sum_{\xi \in L} \frac{1}{\left[ 1+2^k \dist{\xi}{x} \right]^s}
	+ \sum_{\ell =\ceil{c_1}+1}^\infty \sum_{\xi \in L_\ell} \frac{1}{\left[ 1+2^k \dist{\xi}{x} \right]^s} \\
	&\leq \# L + \sum_{\ell=\ceil{c_1}+1}^\infty \sum_{\xi \in L_\ell} \frac{1}{\left[ 1 - c_1\, 2^{k-j} + 2^k \dist{\xi}{\xi_x} \right]^s} \\ 
	&\lesssim 1 + \sum_{\ell=\ceil{c_1}+1}^\infty \frac{\ell^{d-1}}{\left[ 1 + (\ell - \ceil{c_1})\, 2^{k-j} \right]^s} \\
	&\leq 1 + 2^{(j-k)s} \sum_{n=1}^\infty \frac{(n+\ceil{c_1})^{d-1}}{n^s},
\end{align*}
where the last sum converges due to the assumption $s > d$. 
Taking the supremum over all $x\in\Gamma\times\T$ now completes the proof.
\end{proof}

\subsection{Proofs}\label{sect:proofs}
For the sake of completeness, in this final section, we add some proofs which were postponed in order to improve readability of the present manuscript. 
Let us start with showing the result on standard embeddings for sequence spaces $b^{\alpha}_{p,q}(\nabla)$ stated in \autoref{prop:seq-embedding}.

\begin{proof}[Proof (of \autoref{prop:seq-embedding})]
First of all we note that \link{A1}--\link{A3} in  \autoref{ass:struct} assure that every complex sequence $\bm{a}:= (a_{(j,\xi)})_{(j,\xi)\in\nabla}$ can be identified with some $\tilde{\bm{a}} := (\tilde{a}_{(j,\lambda)})_{j\in\N_0,\lambda\in\tilde{\nabla}_j}$
such that for all parameters $\alpha$, $p$, and $q$ we have
\begin{equation}\label{eq:norm-eq}
	\norm{\bm{a}\sep b^{\alpha}_{p,q}(\nabla)} \sim \norm{\tilde{\bm{a}} \sep b^{\alpha}_{p,q}(\tilde{\nabla})}
\end{equation}
(with implied constants solely depending on $p$ and $q$),
where the latter (quasi-)norm is defined by~\link{def:seq_norm} with $\nabla$ replaced by $\tilde{\nabla}:=(\tilde{\nabla}_j)_{j\in\N_0}$ and $\tilde{\nabla}_j \subseteq \Z^d\times\{1,\ldots,2^d-1\}$.
Moreover, $\tilde{\nabla}$ satisfies \link{A4a} or \link{A4b}, respectively, if and only if the same is true for $\nabla$.

In the case of finite index sets our definition of the spaces $b^{\alpha}_{p,q}(\tilde{\nabla})$ exactly matches \cite[Def.~3]{DNS06}. Then for $\gamma \geq 0$ the claimed assertion is covered by \cite[Lem.~4]{DNS06}.
On the other hand, if $\gamma < 0$ then the sequence $\tilde{\bm{a}}^*:=(\tilde{a}^*_{(j,\lambda)})_{j\in\N_0, \lambda\in \tilde{\nabla}_j}$ defined by $\tilde{a}^*_{(j,\lambda)} :\equiv 2^{-j(d/2+\alpha+\gamma/2)}$ for every $\lambda\in\tilde{\nabla}_j$, $j\in\N_0$, belongs to $b_{p_0,q_0}^{\alpha+\gamma}(\tilde{\nabla}) \setminus b_{p_1,q_1}^{\alpha}(\tilde{\nabla})$. 
Using \link{eq:norm-eq} this contradicts $b_{p_0,q_0}^{\alpha+\gamma}(\nabla) \hookrightarrow b_{p_1,q_1}^{\alpha}(\nabla)$ which completes the proof for this case.

Now let the index sets be infinite.
Then, without loss of generality, we may assume that
\begin{equation*}
	\tilde{\nabla}_j 
	= \begin{cases} 
	 	 \Z^d\times\{1\},& \quad j=0,\\
  		 \Z^d\times\{1,\ldots,2^d-1\},& \quad j\in\N,
  	 \end{cases}
\end{equation*}
such that each sequence $\tilde{\bm{a}}$ in the resulting spaces $b_{p,q}^{\alpha}(\tilde{\nabla})$ can be identified with a function $\tilde{f}\in B^\alpha_{p,q}(\R^d)$ by means of the isomorphism constructed in \cite[Thm.~1.64]{T06}. Therein the Besov spaces $B^\alpha_{p,q}(\R^d)$ are defined via harmonic analysis and the mapping $\tilde{\bm{a}} \leftrightarrow \tilde{f}$ is accomplished by the use of Daubechies wavelets (which can be chosen as regular as we want in order to cover arbitrary parameter constellations). Moreover, we conclude
\begin{equation*}
	\norm{\bm{a}\sep b^{\alpha}_{p,q}(\nabla)} \sim \norm{\tilde{\bm{a}} \sep b^{\alpha}_{p,q}(\tilde{\nabla})} \sim \norm{\tilde{f} \sep B^{\alpha}_{p,q}(\R^d)}
\end{equation*}
such that our claim in this case can be derived from corresponding embeddings of classical Besov (function) spaces $B^\alpha_{p,q}(\R^d)$ which are well-known in the literature.
Indeed, at the level of function spaces, sufficiency (and partially also necessity) of our conditions has been proven, e.g., in \cite[Lem.~3]{HanSic2011}.
To show the remaining ``only if'' parts we again construct counter examples at the level of sequence spaces:
If $p_0>p_1$ then there certainly exists $\bm{x}:=(x_\lambda)_{\lambda\in\tilde{\nabla}_0} \in \ell_{p_0}(\tilde{\nabla}_0)\setminus\ell_{p_1}(\tilde{\nabla}_0)$.
Thus the sequence $\tilde{\bm{a}}^*$ defined by $\tilde{a}^*_{(0,\lambda)}:= x_{\lambda}$, $\lambda\in\tilde{\nabla}_0$, and zero otherwise, belongs to 
$b_{p_0,q_0}^{\alpha+\gamma}(\tilde{\nabla}) \setminus b_{p_1,q_1}^{\alpha}(\tilde{\nabla})$. Hence, $p_0\leq p_1$ is necessary.
It remains to check that for $0\leq \gamma < d(1/p_0-1/p_1)$ with $p_0\leq p_1$ the embedding is violated, too. 
For this purpose, we select one $\lambda^*$ in each set $\tilde{\nabla}_j$, $j\in\N_0$, and define $\tilde{\bm{a}}^*:=( \tilde{a}^*_{(j,\lambda)} )_{(j,\lambda)\in\tilde{\nabla}}$ by
\begin{equation*}
	\tilde{a}^*_{(j,\lambda)} := 2^{-j(\alpha+\gamma+d[1/2-1/p_0])} (1+j)^{-2/q_0}
	\qquad \text{for} \quad j\in\N_0 \quad \text{and} \quad \lambda=\lambda^*,
\end{equation*}
and zero otherwise.
Then it can be checked that again $\tilde{\bm{a}}^* \in b_{p_0,q_0}^{\alpha+\gamma}(\tilde{\nabla}) \setminus b_{p_1,q_1}^{\alpha}(\tilde{\nabla})$ which completes the proof.
\end{proof}

It remains to deduce the properties of the three specific wavelet constructions $\Psi_{\mathrm{DS}}$, $\Psi_{\mathrm{HS}}$, and $\Psi_{\mathrm{CTU}}$ stated in \autoref{lem:cancel}.
\begin{proof}[Proof (of \autoref{lem:cancel})]
Since the support conditions \link{P1} and \link{P2} directly follow from the method which defines the wavelet systems, we are left with showing the cancellation-type property \link{P3}. We split its proof into several steps according to the different wavelet constructions under consideration.

\emph{Step 1.}
First of all we deal with the case of (original) composite wavelets $\Psi=\Psi_{\mathrm{DS}}$ as constucted in \cite{DS99} and follow the ideas indicated in \cite[Sect.~4.7]{DS99}.
In order to improve transparency, for this step we stick to the (matrix-vector) notation used therein.
In particular, by $\Phi_j,\tilde{\Phi}_j$ we denote the vectors of primal and dual scaling functions at level $j$ on $\Gamma$. 
For $j\in\N_0$ we make use of the projectors
\begin{equation*}
	Q_j f := \distr{f}{\tilde{\Phi}_j}\Phi_j
	\quad \text{and} \quad
	P_j^\Gamma f := \distr{f}{\tilde{\Lambda}_j}\Phi_j,
\end{equation*}
defined in \cite[Sect.~4.6]{DS99}, which map $L_2(\Gamma)$ and $\mathcal{C}(\Gamma)$ onto $S(\Phi_j)\subset \mathcal{C}(\Gamma)$, respectively.
Here $\tilde{\Lambda}_j$ denotes a vector of certain functionals such that $P_j^\Gamma$ can be represented patchwise as
\begin{equation}\label{eq:P-patchwise}
	(P_j^\Gamma f)\big|_{\Gamma_i} = (P_j^\square(f\circ\kappa_i))\circ\kappa_i^{-1}, 
	\qquad i=1,\ldots,N,
\end{equation}
where $P_j^\square$ are projectors acting on functions on the unit cube $[0,1]^d$; see Section~4.6 and 3.2 in~\cite{DS99} for details.
Denoting the identity by $\Id$ we then have
\begin{equation*}
	Q_j P_j^\Gamma = P_j^\Gamma
	\qquad \text{and} \qquad
	(\Id - Q_j)(\Id-P_j^\Gamma) = \Id - Q_j,
\end{equation*}
because of the duality of $\Phi_j$ and $\tilde{\Phi}_j$.
Clearly $\tilde{\psi}^\Gamma_{j,\xi}\in S(\tilde{\Psi}^\Gamma_j)\,\bot\, S(\Phi_j) \ni Q_j(f)$ such that
\begin{align*}
	\abs{\distr{f}{\tilde{\psi}^\Gamma_{j,\xi}}}
	= \abs{\distr{(\Id-Q_j)f}{\tilde{\psi}^\Gamma_{j,\xi}}}
	= \abs{\distr{(\Id-Q_j)(\Id-P_j^\Gamma)f}{\tilde{\psi}^\Gamma_{j,\xi}}},
\end{align*}
where $\tilde{\Psi}^\Gamma_j=(\tilde{\psi}^\Gamma_{j,\xi})_{\xi\in\nabla_j^\Gamma}$ denotes the vector of dual wavelets at level $j\in\N_0$ on $\Gamma$.
Since the operators $Q_j$ are uniformly bounded on $L_2(\Gamma)$ (cf.\ \cite[Rem.~4.6.1]{DS99}) so is $\Id-Q_j$ and thus the normalization of $\tilde{\psi}^\Gamma_{j,\xi}$ in $L_2(\Gamma)$ gives
\begin{align*}
	\abs{\distr{f}{\tilde{\psi}^\Gamma_{j,\xi}}}
	&\lesssim \norm{(\Id-Q_j)(\Id-P_j^\Gamma)f \sep L_2(\supp \tilde{\psi}^\Gamma_{j,\xi})} \, \norm{\tilde{\psi}^\Gamma_{j,\xi} \sep L_2(\supp \tilde{\psi}^\Gamma_{j,\xi})} \\
	&\lesssim \norm{(\Id-P_j^\Gamma)f \sep L_2(\supp \tilde{\psi}^\Gamma_{j,\xi})} \\
	&\lesssim \sum_{i=1}^N \norm{(\Id-P_j^\Gamma)f \sep L_2(\supp \tilde{\psi}^\Gamma_{j,\xi} \cap \Gamma_i)}.
\end{align*}
Eq.~\link{eq:P-patchwise}, i.e., the patchwise representation of $P_j^\Gamma$, now implies that for every $i=1,\ldots,N$
\begin{align}\label{ineq:Id-P}
	&\norm{(\Id-P_j^\Gamma)f \sep L_2(\supp \tilde{\psi}^\Gamma_{j,\xi} \cap \Gamma_i)} \\
	&\quad \leq \norm{f - p_i\circ\kappa_i^{-1} \sep L_2(\supp \tilde{\psi}^\Gamma_{j,\xi} \cap \Gamma_i)} + 
		\norm{(P_j^\square(f\circ\kappa_i)-p_i)\circ\kappa_i^{-1} \sep L_2(\supp \tilde{\psi}^\Gamma_{j,\xi} \cap \Gamma_i)},\nonumber
\end{align}
where the $p_i$ denote arbitrarily chosen polynomials on $[0,1]^d$.

Let us recall that, by construction, the operators $P_j^\square$, given by
\begin{equation*}
	g\mapsto P_j^\square g := \distr{g}{\tilde{\Lambda}_j^\square}_{L_2([0,1]^d)}\Theta_j^\square,
\end{equation*}
are projections onto $S(\Theta_j^\square)$. Those spaces are strongly related to tensor products of shifts and dilates $\theta_{j,k}$ of (boundary adapted) univariate $D$th-order cardinal B-splines $\theta:={_{D}}{\theta}$, where $D:=D^{\Psi}\in\N$. 
Thus $S(\Theta_j^\square)$ contains the space of all polynomials of total degree less than $D^{\Psi}$ on the unit cube $[0,1]^d$; again see \cite[Sect.~3.2]{DS99}.
The vectors $\tilde{\Lambda}_j^\square$ consist of functionals $\tilde{\lambda}_{j,k}^\square$ which are tensor products of $L_2([0,1])$-inner products (with the duals $\tilde{\theta}_{j,k}$ of $\theta_{j,k}$) and (scaled) point evaluations at the boundary of the interval. 
Remember that it suffices to assume that $g\in H^s([0,1]^d)$ with $s>d/2$ in order to ensure that sampling of the function $g$ at points on the boundary of the unit cube is well-defined (in this case we find a continuous representer of $g$ by means of Sobolev's embedding theorem).

Hence, if we restrict ourselves to polynomials $p_i$ of degree smaller than $D^{\Psi}$ then $P_j^\square p_i = p_i$ and \link{ineq:Id-P} can be rewritten as
\begin{align}
	&\norm{(\Id-P_j^\Gamma)f \sep L_2(\supp \tilde{\psi}^\Gamma_{j,\xi} \cap \Gamma_i)} \nonumber\\
	&\qquad \leq \norm{g_i\circ\kappa_i^{-1} \sep L_2(\supp \tilde{\psi}^\Gamma_{j,\xi} \cap \Gamma_i)} + 
		\norm{P_j^\square g_i\circ\kappa_i^{-1} \sep L_2(\supp \tilde{\psi}^\Gamma_{j,\xi} \cap \Gamma_i)}, \label{two_summands}
\end{align}
where we set $g_i:=f\circ\kappa_i - p_i$.
In order to bound the second summand in \link{two_summands} we define the index sets
\begin{equation*}
	I_{j,\xi}^{\, i} 
	:= \left\{ k \sep \supp \theta_{j,k}^\square \cap \kappa_i^{-1}(\supp \tilde{\psi}^\Gamma_{j,\xi} \cap \Gamma_i) \neq \emptyset \right\}, \qquad i=1,\ldots,N,
\end{equation*}
for all tensor products $\theta_{j,k}^\square$ in $\Theta_j^\square$ whose support hit the set $\kappa_i^{-1}(\supp \tilde{\psi}^\Gamma_{j,\xi} \cap \Gamma_i)$ in $[0,1]^d$. Due to the local support of the $\theta_{j,k}^\square$ the cardinality of this index sets is uniformly bounded in $j$, $\xi$, and~$i$.
Therefore we conclude that
\begin{align*}
	\norm{P_j^\square g_i\circ\kappa_i^{-1} \sep L_2(\supp \tilde{\psi}^\Gamma_{j,\xi} \cap \Gamma_i)}
	&= \norm{P_j^\square g_i \sep L_2(\kappa_i^{-1}(\supp \tilde{\psi}^\Gamma_{j,\xi} \cap \Gamma_i))} \\
	&= \norm{\distr{g_i}{\tilde{\Lambda}_j^\square}_{L_2([0,1]^d)}\Theta_j^\square \sep L_2(\kappa_i^{-1}(\supp \tilde{\psi}^\Gamma_{j,\xi} \cap \Gamma_i))}
\end{align*}	
can be estimated from above by 
\begin{align*}
	 \norm{ \sum_{k\in I_{j,\xi}^{\, i}} \distr{g_i}{\tilde{\lambda}_{j,k}^\square}_{L_2([0,1]^d)}\theta_{j,k}^\square \sep L_2([0,1]^d) }
	\lesssim \max_{k\in I_{j,\xi}^{\, i}} \abs{\distr{g_i}{\tilde{\lambda}_{j,k}^\square}_{L_2([0,1]^d)}} \norm{\theta_{j,k}^\square \sep L_2([0,1]^d) }.
\end{align*}
Using the normalization of $\theta_{j,k}^\square$, as well as the bound on $\tilde{\lambda}_{j,k}^\square(g_i):=\distr{g_i}{\tilde{\lambda}_{j,k}^\square}_{L_2([0,1]^d)}$ stated in \cite[Ineq.~(3.2.5)]{DS99}, we obtain
\begin{align*}
	\norm{P_j^\square g_i\circ\kappa_i^{-1} \sep L_2(\supp \tilde{\psi}^\Gamma_{j,\xi} \cap \Gamma_i)}
	&\lesssim 2^{-j d/2} \max_{k\in I_{j,\xi}^{\, i}} \norm{g_i \sep L_\infty(\supp \tilde{\theta}_{j,k}^\square)} \\
	&\leq 2^{-j d/2} \norm{g_i \sep L_\infty(\tilde{C}^{\,i}_{j,\xi})},
\end{align*}
where $\tilde{C}^{\,i}_{j,\xi}$ denotes a cube in $[0,1]^d$ that contains $\kappa_i^{-1}(\supp \tilde{\psi}^\Gamma_{j,\xi} \cap \Gamma_i) \cup ( \bigcup_{k\in I_{j,\xi}^{\, i}} \supp \tilde{\theta}_{j,k}^\square )$.
Since, by construction, we have $\diam(\supp \theta_{j,k}^\square \cup \supp \tilde{\theta}_{j,k}^\square)\lesssim 2^{-j}$, this cube can be chosen such that $\abs{ \tilde{C}^{\,i}_{j,\xi} } \lesssim 2^{-jd}$.

Moreover, our choice of $\tilde{C}^{\,i}_{j,\xi}$ allows to bound the first summand in \link{two_summands} by the same quantity:
\begin{align}
	\norm{g_i\circ\kappa_i^{-1} \sep L_2(\supp \tilde{\psi}^\Gamma_{j,\xi} \cap \Gamma_i)}
	&\leq \norm{g_i \sep L_2(\tilde{C}^{\,i}_{j,\xi})} \nonumber\\
	&\leq \abs{\tilde{C}^{\,i}_{j,\xi}}^{1/2} \cdot \norm{g_i \sep L_\infty(\tilde{C}^{\,i}_{j,\xi})}\nonumber\\
	&\lesssim 2^{-j d/2} \norm{g_i \sep L_\infty(\tilde{C}^{\,i}_{j,\xi})}. \label{est:L2Linfty}
\end{align}

In conclusion, for $(f\circ\kappa_i)\in H^s(\tilde{C}^{\,i}_{j,\xi})$ with $s>d/2$, we established the bound
\begin{equation}\label{est:Linfty}
	\abs{\distr{f}{\tilde{\psi}_{j,\xi}}}
	\lesssim \sum_{i=1}^N 2^{-j d/2} \norm{g_i \sep L_\infty(\tilde{C}^{\,i}_{j,\xi})}
	= \sum_{i=1}^N 2^{-j d/2} \norm{f\circ\kappa_i - p_i \sep L_\infty(\tilde{C}^{\,i}_{j,\xi})},
\end{equation}
where the $p_i$ are arbitrary polynomials of degree less than $D^{\Psi}$ on $[0,1]^d$.

In order to show the desired estimate \link{patchwise_bound} we finally apply Whitney's bound which (adapted to our needs) takes the form
\begin{equation*}
	\inf_{p \in \Pi_{\ceil{s}}(\Omega)} \norm{F-p \sep L_\infty(\Omega)}
	\lesssim \abs{\Omega}^{s/d-1/2} \, \abs{F}_{H^s(\Omega)}, 
	\qquad s \geq  d/2.
\end{equation*}
Therein $\Omega$ denotes some cube in $\R^d$ with volume $\abs{\Omega}$, $\Pi_{\ceil{s}}(\Omega)$ is the space of polynomials of degree less than $\ceil{s}$ (the smallest integer larger or equal to $s$), and $ \abs{F}_{H^s(\Omega)}$ is the usual semi-norm of $F$ in $H^s(\Omega)$.
Accordingly,
\begin{equation*}
	\abs{\distr{f}{\tilde{\psi}_{j,\xi}}}
	\lesssim \sum_{i=1}^N 2^{-j d/2} 2^{-j d (s/d-1/2)} \, \abs{f\circ\kappa_i}_{H^s(\tilde{C}^{\,i}_{j,\xi})},
	\qquad d/2 < s \leq D^{\Psi},
\end{equation*}
which proves the first assertion claimed in \link{P3} for $\Psi=\Psi_{\mathrm{DS}}$.

The proof of the second part of \link{P3} is obtained by exactly the same arguments, where every primal quantity (denoted without tilde) is replaced by its dual analogue (with tilde) and vice versa.

\emph{Step 2.} We turn to modified composite wavelets $\Psi=\Psi_{\mathrm{HS}}$ as constructed in \cite{HS06}. 
In this case we can make use of the fact that the primal and the dual wavelets satisfy patchwise cancellation properties of order $\tilde{D}^{\Psi}$ and $D^{\Psi}$, respectively. 
That is, we have
\begin{equation}\label{patchwise_cancel}
	\distr{\psi_{j,\xi}^{\Gamma} \circ \kappa_i}{\, p_i}_{L_2([0,1]^d)}=0
	\qquad \text{and} \qquad
	\distr{\tilde{\psi}_{j,\xi}^{\Gamma} \circ \kappa_i}{\, p_i}_{L_2([0,1]^d)}=0, \qquad i=1,\ldots,N,
\end{equation}
for all $j$ and $\xi$, as well as every polynomial $p_i$ of degree less than $\tilde{D}^{\Psi}$ resp. $D^{\Psi}$ on $[0,1]^d$.
For the primal wavelets this has been shown in \cite[Prop.~4.1]{HS06}, whereas the assertion at the dual side simply follows from biorthogonality of the dual wavelets with the primal scaling functions (which are exact of order $D^{\Psi}$).

Now the derivation of the first part of \link{P3} for $\Psi=\Psi_{\mathrm{HS}}$ is straightforward. 
From the definition of the inner product $\distr{\cdot}{\cdot}$ and the patchwise cancellation property \link{patchwise_cancel} we deduce
\begin{align*}
	\abs{\distr{f}{\tilde{\psi}^\Gamma_{j,\xi}}}
	\leq \sum_{i=1}^N \abs{\distr{\tilde{\psi}_{j,\xi}^\Gamma\circ\kappa_i}{f\circ \kappa_i}_{L_2(\widetilde{C}_{j,\xi}^i)}}
	= \sum_{i=1}^N \abs{\distr{\tilde{\psi}_{j,\xi}^\Gamma\circ\kappa_i}{f\circ \kappa_i-p_i}_{L_2(\widetilde{C}_{j,\xi}^i)}}
\end{align*}
for all $p_i \in \Pi_{D^{\Psi}}(\widetilde{C}_{j,\xi}^i)$, where the cubes $\widetilde{C}_{j,\xi}^i \subset [0,1]^d$, $i=1,\ldots,N$, can be chosen according to the requirements in the statement of the lemma.
Again we set $g_i := f\circ\kappa_i-p_i$, apply Cauchy-Schwarz, and use the $L_2(\Gamma)$-normalization of the dual wavelets to obtain
\begin{align*}
	\abs{\distr{\tilde{\psi}_{j,\xi}^\Gamma\circ\kappa_i}{f\circ \kappa_i-p_i}_{L_2(\widetilde{C}_{j,\xi}^i)}}
	\leq \norm{g_i \sep L_2(\widetilde{C}_{j,\xi}^i)}
	\lesssim 2^{-j d/2} \norm{g_i \sep L_\infty(\widetilde{C}_{j,\xi}^i)},
\end{align*}
where the second estimate is derived as in \link{est:L2Linfty}. Hence, again we have shown \link{est:Linfty} and (as in the previous step) the bound \link{patchwise_bound} is implied by Whitney's estimate.

Since the proof for the primal wavelets is obtained in the same way, we have shown the claim also for this case.

\emph{Step 3.}
The proof of \link{P3} for the construction $\Psi=\Psi_{\mathrm{CTU}}$ given in \cite{CTU99} can be performed literally as in the previous step. The needed patchwise cancellation property \link{patchwise_cancel} for the primal wavelets can be found as Formula (3.12) in \cite[Section 3.4.1]{CTU00}.
\end{proof}

\section*{Acknowledgements}
\addcontentsline{toc}{section}{Acknowledgments}
The author likes to thank S.~Dahlke and H.~Harbrecht for several valuable discussions.

\addcontentsline{toc}{chapter}{References}
\small

\begin{thebibliography}{10}

\bibitem{BL76}
J.~Bergh and J.~L{\"o}fstr{\"o}m.
\newblock {\em {Interpolation {S}paces. {A}n {I}ntroduction}}.
\newblock Grundlehren der mathematischen Wissenschaften. Springer-Verlag,
  Berlin, 1976.

\bibitem{CTU99}
C.~Canuto, A.~Tabacco, and K.~Urban.
\newblock The wavelet element method part~{I}: {C}onstruction and analysis.
\newblock {\em Appl. Comput. Harmon. Anal.}, 6\penalty0 (1):\penalty0 1--52,
  1999.

\bibitem{CTU00}
C.~Canuto, A.~Tabacco, and K.~Urban.
\newblock The wavelet element method part~{II}: {R}ealization and additional
  features in {2D} and {3D}.
\newblock {\em Appl. Comput. Harmon. Anal.}, 8\penalty0 (2):\penalty0 123--165,
  2000.

\bibitem{Coh2003}
A.~Cohen.
\newblock {\em Numerical {A}nalysis of {W}avelet {M}ethods}, volume~32 of {\em
  Studies in Mathematics and its Applications}.
\newblock North-Holland, Amsterdam, 2003.

\bibitem{CDD01}
A.~Cohen, W.~Dahmen, and R.~A. DeVore.
\newblock Adaptive wavelet methods for elliptic operator equations:
  {C}onvergence rates.
\newblock {\em Math. Comp.}, 70:\penalty0 27--75, 2001.

\bibitem{CDD02}
A.~Cohen, W.~Dahmen, and R.~A. DeVore.
\newblock Adaptive wavelet methods {II}: {B}eyond the elliptic case.
\newblock {\em J. Found. Comput. Math.}, 2\penalty0 (3):\penalty0 203--245,
  2002.

\bibitem{CohMas2000}
A.~Cohen and R.~Masson.
\newblock Wavelet adaptive method for second order elliptic problems:
  {B}oundary conditions and domain decomposition.
\newblock {\em Numer. Math.}, 86\penalty0 (2):\penalty0 193--238, 2000.

\bibitem{DahDahDeV1997}
S.~Dahlke, W.~Dahmen, and R.~A. DeVore.
\newblock Nonlinear approximation and adaptive techniques for solving elliptic
  operator equations.
\newblock In W.~Dahmen, A.~Kurdila, and P.~Oswald, editors, {\em {M}ultsicale
  {W}avelet {M}ethods for {P}artial {D}ifferential {E}quations}, pages
  237--283, San Diego, 1997. Academic Press.

\bibitem{DahDeV1997}
S.~Dahlke and R.~A. DeVore.
\newblock Besov regularity for elliptic boundary value problems.
\newblock {\em Comm. Partial Differential Equations}, 22\penalty0
  (1-2):\penalty0 1--16, 1997.

\bibitem{DahDieHar+2014}
S.~Dahlke, L.~Diening, C.~Hartmann, B.~Scharf, and M.~Weimar.
\newblock Besov regularity of solutions to the $p$-{P}oisson equation.
\newblock In preparation, 2014.

\bibitem{DNS06}
S.~Dahlke, E.~Novak, and W.~Sickel.
\newblock Optimal approximation of elliptic problems by linear and nonlinear
  mappings {II}.
\newblock {\em J.~Complexity}, 22\penalty0 (4):\penalty0 549--603, 2006.

\bibitem{DahRaaWer+2007}
S.~Dahlke, T.~Raasch, M.~Werner, M.~Fornasier, and R.~Stevenson.
\newblock Adaptive frame methods for elliptic operator equations: {T}he
  steepest descent approach.
\newblock {\em IMA J. Numer. Anal.}, 27\penalty0 (4):\penalty0 717--740, 2007.

\bibitem{DahWei2013}
S.~Dahlke and M.~Weimar.
\newblock Besov regularity for operator equations on patchwise smooth
  manifolds.
\newblock Report 2013-03, Fachbereich {M}athematik und {I}nformatik,
  {P}hilipps-{U}niversit{\"a}t {M}arburg, 2013.

\bibitem{Dah1997}
W.~Dahmen.
\newblock Wavelet and multiscale methods for operator equations.
\newblock {\em Acta Numer.}, 6:\penalty0 55--228, 1997.

\bibitem{DahHarSch2006}
W.~Dahmen, H.~Harbrecht, and R.~Schneider.
\newblock Compression techniques for boundary integral equations.
  {A}symptotically optimal complexity estimates.
\newblock {\em SIAM J. Numer. Anal.}, 43\penalty0 (6):\penalty0 2251--2271,
  2006.

\bibitem{DahHarSch2007}
W.~Dahmen, H.~Harbrecht, and R.~Schneider.
\newblock Adaptive methods for boundary integral equations: {C}omplexity and
  convergence estimates.
\newblock {\em Math. Comp.}, 76\penalty0 (259):\penalty0 1243--1274, 2007.

\bibitem{DahSchn1998}
W.~Dahmen and R.~Schneider.
\newblock Wavelets with complementary boundary conditions --- {F}unction spaces
  on the cube.
\newblock {\em Result. Math.}, 34\penalty0 (3--4):\penalty0 255--293, 1998.

\bibitem{DS99}
W.~Dahmen and R.~Schneider.
\newblock Composite wavelet bases for operator equations.
\newblock {\em Math. Comp.}, 68:\penalty0 1533--1567, 1999.

\bibitem{DahSch1999}
W.~Dahmen and R.~Schneider.
\newblock Wavelets on manifolds {I}: {C}onstruction and domain decomposition.
\newblock {\em SIAM J. Math. Anal.}, 31\penalty0 (1):\penalty0 184--230, 1999.

\bibitem{DeV1998}
R.~A. DeVore.
\newblock Nonlinear approximation.
\newblock {\em Acta Numer.}, 7:\penalty0 51--150, 1998.

\bibitem{DJP92}
R.~A. DeVore, B.~Jawerth, and V.~Popov.
\newblock Compression of wavelet decompositions.
\newblock {\em Am. J. Math.}, 114\penalty0 (4):\penalty0 737--785, 1992.

\bibitem{Fo95}
G.~B. Folland.
\newblock {\em Introduction to {P}artial {D}ifferential {E}quations (2nd ed.)}.
\newblock Princeton University Press, Princeton, NJ, 1995.

\bibitem{FJ90}
M.~Frazier and B.~Jawerth.
\newblock A discrete transform and decomposition of distribution spaces.
\newblock {\em J.~Funct.~Anal.}, 93\penalty0 (1):\penalty0 34--170, 1990.

\bibitem{Han2012}
M.~Hansen.
\newblock {$N$}-term approximation rates and {B}esov regularity for elliptic
  {PDE}s on polyhedral domains.
\newblock Technical Report 2012-41, Seminar for Applied Mathematics, ETH
  Z{\"u}rich, 2012.
\newblock To appear in J. Found. Comput. Math., 2014.

\bibitem{HanSic2011}
M.~Hansen and W.~Sickel.
\newblock Best $m$-term approximation and {L}izorkin-{T}riebel spaces.
\newblock {\em J.~Approx. Theory}, 163\penalty0 (8):\penalty0 923--954, 2011.

\bibitem{HarSch2004}
H.~Harbrecht and R.~Schneider.
\newblock Biorthogonal wavelet bases for the boundary element method.
\newblock {\em Math. Nachr.}, 269--270:\penalty0 167--188, 2004.

\bibitem{HS06}
H.~Harbrecht and R.~Stevenson.
\newblock Wavelets with patchwise cancellation properties.
\newblock {\em Math. Comp.}, 75:\penalty0 1871--1889, 2006.

\bibitem{KMM07}
N.~Kalton, S.~Mayboroda, and M.~Mitrea.
\newblock Interpolation of {H}ardy-{S}obolev-{B}esov-{T}riebel-{L}izorkin
  spaces and applications to problems in partial differential equations.
\newblock In L.~{De Carli} and M.~Milman, editors, {\em Interpolation {T}heory
  and {A}pplications ({C}ontemporary {M}athematics 445)}, pages 121--177,
  Providence, RI, 2007. Amer. Math. Soc.

\bibitem{Lem1989}
P.~G. Lemari\'{e}.
\newblock Bases d'ondelettes sur les groupes de lie statisfii\'{e}s.
\newblock {\em Bull. Soc. math. France}, 117:\penalty0 211--232, 1989.

\bibitem{Pri2010}
M.~Primbs.
\newblock New stable biorthogonal spline-wavelets on the interval.
\newblock {\em Result. Math.}, 57\penalty0 (1--2):\penalty0 121--162, 2010.

\bibitem{RS96}
T.~Runst and W.~Sickel.
\newblock {\em {S}obolev {S}paces of {F}ractional {O}rder, {N}emytskij
  {O}perators and {N}onlinear {P}artial {D}ifferential {E}quations}.
\newblock de Gruyter, Berlin, 1996.

\bibitem{Ry99}
V.~S. Rychkov.
\newblock On a theorem of {B}ui, {P}aluszy\'nski, and {T}aibleson.
\newblock In {\em Studies on the theory of differentiable functions of several
  variables and its applications. Part 18}, pages 280--292. Moskva: Nauka;
  Moskva: MAIK Nauka/Interperiodika, 1999.

\bibitem{T06}
H.~Triebel.
\newblock {\em Theory of {F}unction {S}paces III}.
\newblock Birkh\"auser, Basel, 2006.

\bibitem{T08}
H.~Triebel.
\newblock {\em Function {S}paces and {W}avelets on {D}omains}, volume~7 of {\em
  EMS Tracts in Mathematics}.
\newblock European Mathematical Society (EMS), Z\"urich, 2008.

\bibitem{T83}
H.~Triebel.
\newblock {\em Theory of {F}unction {S}paces}.
\newblock Birkh\"auser, Basel, reprint of the 1983 original edition, 2010.

\end{thebibliography}

\end{document}